\documentclass[12pt]{amsart}

\usepackage[colorlinks=true]{hyperref}
\usepackage{xypic}
\usepackage{amssymb}
\usepackage{stmaryrd}

\renewcommand{\d}{\mathrm{d}}
\newcommand{\ddr}{\mathrm{d}_{\mathrm{dR}}}
\newcommand{\D}{\mathrm{D}}
\newcommand{\T}{\mathrm{T}}
\newcommand{\U}{\mathrm{U}}
\newcommand{\Z}{\mathrm{Z}}

\newcommand{\bE}{\mathbb{E}}
\newcommand{\bL}{\mathbb{L}}
\newcommand{\bP}{\mathbb{P}}
\newcommand{\coP}{\mathrm{co}\mathbb{P}}
\newcommand{\bT}{\mathbb{T}}

\newcommand{\g}{\mathfrak{g}}

\newcommand{\Alg}{\mathrm{Alg}}
\newcommand{\Ass}{\mathrm{Ass}}
\newcommand{\coAss}{\mathrm{coAss}}
\newcommand{\BD}{\mathbb{BD}}
\newcommand{\Br}{\mathrm{Br}}
\newcommand{\coBD}{\mathrm{co}\mathbb{BD}}
\newcommand{\C}{\mathrm{C}}

\newcommand{\cop}{\mathrm{cop}}
\newcommand{\CC}{\mathrm{CC}}
\newcommand{\Der}{\mathrm{Der}}

\newcommand{\Hom}{\mathrm{Hom}}
\newcommand{\id}{\mathrm{id}}
\newcommand{\op}{\mathrm{op}}
\newcommand{\Pol}{\mathrm{Pol}}
\newcommand{\sgn}{\mathrm{sgn}}

\DeclareMathOperator{\pt}{pt}
\DeclareMathOperator{\eq}{eq}
\DeclareMathOperator{\Spec}{Spec}
\DeclareMathOperator{\Sym}{Sym}

\newcommand{\llpar}{(\!(}
\newcommand{\rrpar}{)\!)}

\newtheorem*{theorem}{Theorem}
\newtheorem{thm}{Theorem}[section]
\newtheorem{prop}[thm]{Proposition}
\newtheorem{cor}[thm]{Corollary}

\theoremstyle{definition}
\newtheorem{defn}[thm]{Definition}

\theoremstyle{remark}
\newtheorem{remark}[thm]{Remark}
\newtheorem{example}[thm]{Example}

\title{Poisson reduction as a coisotropic intersection}

\author{Pavel Safronov}
\address{Mathematical Institute, Radcliffe Observatory Quarter, Woodstock Road, Oxford UK, OX2 6GG}
\curraddr{Institut f\"{u}r Mathematik, Universit\"{a}t Z\"{u}rich, Winterthurerstrasse 190, 8057 Z\"{u}rich, Switzerland}
\email{pavel.safronov@math.uzh.ch}

\begin{document}

\maketitle
\begin{abstract}
We give a definition of coisotropic morphisms of shifted Poisson (i.e. $\bP_n$) algebras which is a derived version of the classical notion of coisotropic submanifolds. Using this we prove that an intersection of coisotropic morphisms of shifted Poisson algebras carries a Poisson structure of shift one less. Using an interpretation of Hamiltonian spaces as coisotropic morphisms we show that the classical BRST complex computing derived Poisson reduction coincides with the complex computing coisotropic intersection. Moreover, this picture admits a quantum version using brace algebras and their modules: the quantum BRST complex is quasi-isomorphic to the complex computing tensor product of brace modules.
\end{abstract}

\section*{Introduction}

The goal of the present paper is to introduce the notion of a coisotropic structure on shifted Poisson algebras on the level of 1-categories and show that it satisfies some expected properties such as:
\begin{itemize}
\item Moment maps provide examples of coisotropic structures,

\item A derived intersection $B_1\otimes^\bL_A B_2$ of coisotropic maps $A\rightarrow B_1$ and $A\rightarrow B_2$, where $A$ is an $n$-shifted Poisson algebra, carries an $(n-1)$-shifted Poisson structure up to homotopy.
\end{itemize}

The homotopy theory of such coisotropic structures is further studied in \cite{MS16} and \cite{MS17}.

\subsection*{Coisotropic intersections}

Motivated by Lagrangian Floer theory and Donaldson--Thomas theory, Behrend and Fantechi \cite{BF} showed that the cohomology of the algebra of functions on a derived intersection of two holomorphic Lagrangian submanifolds of a complex symplectic manifold carries a $(-1)$-shifted Poisson ($\bP_0$) structure.

Pantev, To\"{e}n, Vaqui\'{e} and Vezzosi \cite{PTVV} gave a derived-geometric interpretation of this result. Namely, it was shown that a derived intersection of two algebraic Lagrangians carries a $(-1)$-shifted symplectic structure. More generally, they have shown that a derived intersection of two Lagrangians in an $n$-shifted symplectic stack is $(n-1)$-shifted symplectic.

Baranovsky and Ginzburg \cite{BG} generalized the Behrend--Fantechi result in a different direction. Namely, they have shown that the cohomology of the algebra of functions on the derived intersection of two coisotropic subvarieties of a Poisson variety carries a $\bP_0$-structure. It is thus natural to ask whether one can lift the Baranovsky--Ginzburg construction to the chain level.

Calaque, Pantev, To\"{e}n, Vaqui\'{e} and Vezzosi \cite{CPTVV} introduced $n$-shifted Poisson structures on derived stacks and derived coisotropic structures on morphisms of stacks. Let us recall their definitions in the affine setting. Let $A$ be a commutative dg algebra. By a theorem of Melani \cite{Mel}, an $n$-shifted Poisson structure on $A$ is the same as a $\bP_{n+1}$-structure on $A$, i.e. a Poisson bracket of cohomological degree $-n$. For $B$ another commutative dg algebra, CPTVV define a coisotropic structure on a morphism $A\rightarrow B$ to be the same as a $\bP_n$-structure on $B$ together with the data of an associative action of $A$ on $B$ in the category of $\bP_n$-algebras. To define such a notion, they use a result announced by Rozenblyum (Poisson additivity) which identifies $\bP_{n+1}$-algebras with associative algebra objects in the $\infty$-category of $\bP_n$-algebras. This definition is expected to give rather easily a $\bP_n$-structure on a coisotropic intersection. However, Poisson additivity is not given by explicit formulas, so the explicit Poisson structure on the coisotropic intersection would be difficult to write down.

In this paper we develop coisotropic structures in the affine setting, i.e. for arbitrary commutative differential graded algebras. We model an action of the $\bP_{n+1}$-algebra $A$ on a $\bP_n$-algebra $B$ by a $\bP_{n+1}$-morphism $A\rightarrow \Z(B)$ (Definition \ref{defn:coisotropic}). Here
\[\Z(B)=\Hom_B(\Sym_B(\Omega^1_B[n]), B)\]
is the complex of $(n-1)$-shifted polyvector fields with the differential twisted by the Poisson structure on $B$ which is a derived version of the Poisson center of $B$.

Note that the $\bP_{n+1}$-structure on $\Z(B)$ is very explicit: it is given by the Schouten bracket (i.e. by the commutator of multiderivations). Using this definition we prove the following theorem (Theorem \ref{thm:coisotropicintersection}).

\begin{theorem}
Let $A$ be a $\bP_{n+1}$-algebra and $A\rightarrow B_1,\ A\rightarrow B_2$ two coisotropic morphisms. Then the derived intersection $B_1\otimes^{\bL}_A B_2$ carries a homotopy $\bP_n$-structure. Moreover, the natural projection $B_1^{\op}\otimes B_2\rightarrow B_1\otimes^{\bL}_A B_2$ is a $\bP_n$-morphism where $B_1^{\op}$ denotes the same commutative dg algebra with the opposite Poisson bracket.
\end{theorem}

The proof of this theorem uses ideas from Koszul duality. Since one can identify by Poisson additivity a $\bP_{n+1}$-algebra with an associative algebra object in $\bP_n$-algebras, one expects the Koszul dual coalgebra of a $\bP_{n+1}$-algebra to carry a compatible $\bP_n$-structure; indeed, it is given by explicit formulas using the bar complex (Proposition \ref{prop:PnKoszul}). Similarly, we show that the Koszul dual to the $A$-module $B_i$ carries a homotopy $\bP_n$-structure given by the coisotropic structure. Finally, the derived tensor product $B_1\otimes^{\bL}_A B_2$ can be written as an underived cotensor product on the Koszul dual side.

After the present paper was posted on the ArXiv, the proof of Poisson additivity was written down in \cite{Saf2}. In the same paper it was shown that our definition of coisotropic morphisms is equivalent to the one of \cite{CPTVV}.

\subsection*{Moment maps}

We give an application of derived coisotropic intersection to Hamiltonian reduction.

Let us recall that given a symplectic manifold $X$ with a $G$-action preserving the symplectic form, a moment map is a $G$-equivariant morphism $\mu\colon X\rightarrow\g^*$ which is a Hamiltonian for the $G$-action. Hamiltonian reduction is defined to be the quotient
\[X//G=\mu^{-1}(0)/G.\]

If $0$ is a regular value for $\mu$ and the $G$-action on $\mu^{-1}(0)$ is free and proper, the quotient is a symplectic manifold as shown by Marsden and Weinstein \cite{MW}. If one of these conditions fails, the quotient is only a stratified symplectic manifold which hints that it is a shadow of a derived symplectic structure.

Indeed, passing to the setting of derived algebraic geometry we can rewrite
\[X//G\cong  \pt/G \times_{\g^*/G} X/G.\]

Moreover, as shown in \cite{Cal} and \cite{Saf}, Hamiltonian $G$-spaces are the same as Lagrangians in the $1$-shifted symplectic stack $\g^*/G$. Therefore, $X//G$ is a Lagrangian intersection and so carries a derived symplectic structure.

In this paper we show similar statements in the affine Poisson setting. Namely, if $B$ is a $\bP_1$-algebra (a dg Poisson algebra) with $\mu\colon \Sym\g\rightarrow B$ a moment map for a $\g$-action on $B$ we show that the induced morphism \[\C^\bullet(\g, \Sym\g)\rightarrow \C^\bullet(\g, B)\]
is coisotropic. Here $\C^\bullet(\g, -)$ is the Chevalley--Eilenberg cochain complex and  $\C^\bullet(\g, \Sym\g)$ is the $\bP_2$-algebra (i.e. Gerstenhaber algebra) of functions on the quotient $\g^*/G$ with $G$ formal.

The coisotropic intersection
\[\C^\bullet(\g, k)\otimes^{\bL}_{\C^\bullet(\g, \Sym\g)} \C^\bullet(\g, B)\]
is thus a derived Poisson reduction which we show to be quasi-isomorphic (as a commutative dg algebra) to the classical BRST complex as defined by Kostant and Sternberg \cite{KS}.

Let us note that this perspective on Poisson reduction is somewhat orthogonal to the one obtained by computing coisotropic reduction of $\mu^{-1}(0)\subset X$ using the BFV complex (see e.g. \cite{Sta}). Indeed, in that approach one considers a coisotropic \emph{reduction} of the 0-shifted coisotropic morphism $\mu^{-1}(0)\rightarrow X$. On the other hand, in our approach we consider a coisotropic \emph{intersection} of the 1-shifted coisotropic morphism $X/G\rightarrow \g^*/G$. The precise relationship between the two approaches is not clear to the author.

\subsection*{Quantization}

We also develop quantum versions of our results in the sense of deformation quantization. Namely, while deformation quantizations of $\bP_1$-algebras are dg algebras, deformation quantizations of $\bP_2$-algebras are $\bE_2$-algebras, i.e. algebras over the operad of little disks, which we model by brace algebras following \cite{MS}. We introduce a notion of a brace module $M$ over a brace algebra $A$ which provides deformation quantization of the notion of a coisotropic morphism from a $\bP_2$-algebra $A$ to a $\bP_1$-algebra $M$. One way to think of it is as follows: the pair (brace algebra, brace module) is conjectured to be the same as an algebra over the Swiss-cheese operad introduced by Voronov \cite{Vor}. We prove the following quantum version of the coisotropic intersection theorem (Theorem \ref{thm:quantumintersection}).

\begin{theorem}
Let $A$ be a brace algebra, $B_1$ a left brace module and $B_2$ a right brace module over $A$. Then the derived tensor product $B_1\otimes^{\bL}_A B_2$ carries a natural dg algebra structure such that the projection $B_1^{\op}\otimes B_2\rightarrow B_1\otimes^{\bL}_A B_2$ is an algebra morphism, where $B_1^{\op}$ is the algebra with the opposite multiplication.
\end{theorem}

We apply this result to quantum moment maps. Recall that a quantum moment map is given by a morphism of associative algebras $\U\g\rightarrow B$, where $B$ is an associative algebra. These are to be thought of as deformation quantizations of Poisson maps $\Sym\g\rightarrow B$ (classical moment map), where $B$ is a Poisson algebra.

A quantization of the $\bP_2$-algebra $\C^\bullet(\g, \Sym\g)$ is the brace algebra $\CC^\bullet(\U\g, \U\g)$, the Hochschild cochain complex of the universal enveloping algebra $\U\g$. We show that a quantum moment map $\U\g\rightarrow B$ makes $\CC^\bullet(\U\g, B)$ into a brace module over $\CC^\bullet(\U\g, \U\g)$. The tensor product
\[\CC^\bullet(\U\g, k)\otimes^{\bL}_{\CC^\bullet(\U\g, \U\g)} \CC^\bullet(\U\g, B)\]
computing derived quantum Hamiltonian reduction is therefore a dg algebra which is shown to be quasi-isomorphic to the quantum BRST complex \cite{KS}.

This point of view on quantum Hamiltonian reduction allows one to generalize ordinary (i.e. $\bE_1$) Hamiltonian reduction to $\bE_n$-algebras (algebras over the operad of little $n$-disks) which we sketch in Section \ref{sect:Enreduction}.

Both classical and quantum constructions can be put on the same footing if one starts with a deformation quantization for which we use the language of Beilinson--Drinfeld algebras \cite[Section 2.4]{CG}. We end the paper with some theorems that interpolate between classical coisotropic intersections and tensor products of brace modules.

\subsection*{Acknowledgements}

The author would like to thank Thel Seraphim for collaboration at an early stage of the project and the referee for many useful comments. This research was supported by the EPSRC grant EP/I033343/1.

\subsection*{Notation}

We work over a field $k$ of characteristic zero. We adopt the cohomological grading convention. By a dga we mean a differential graded algebra over $k$ not necessarily non-positively graded. For $A$ a dga and $M$ and $N$ two modules we denote by $M\otimes^{\bL}_A N$ the resolution given by the two-sided bar complex.

An $(n,m)$-shuffle $\sigma\in S_{n,m}$ is a permutation $\sigma\in S_{n+m}$ such that $\sigma(1)< \dots< \sigma(n)$ and $\sigma(n+1) < \dots < \sigma(n+m)$.

\section{Shifted Poisson algebras}

\label{sect:Pnalgebras}

\subsection{Polyvector fields}

Let $A$ be a cdga. We denote by $\T_A=\Der(A, A)$ the $A$-module of derivations which is a dg Lie algebra over $k$. We define the complex of $(n-1)$-shifted polyvector fields to be
\[\Pol(A, n-1) = \Hom_A(\Sym_A(\Omega^1_A[n]), A).\]
$\Pol(A, n-1)$ has a natural \emph{weight} grading under which $\Omega^1_A$ has weight $-1$ and we can decompose
\[\Pol(A, n-1) = \bigoplus_k \Pol(A, n-1)^k = \bigoplus_k \Hom_A(\Sym^k_A(\Omega^1_A[n]), A).\]

We denote by $\lrcorner$ the natural duality pairing between $\Pol(A, n-1)$ and $\Sym_A(\Omega^1[n])$. Given a polyvector $v\in \Pol(A, n-1)^k$ we define
\begin{equation}
v(a_1, \dots, a_k) = v\lrcorner (\ddr\otimes \dots\otimes \ddr)(a_1\otimes \dots\otimes a_k),
\label{eq:decalage}
\end{equation}
where the formal symbol $\ddr$ is put in degree $-n$ to fix the signs and $a_i\in A$. The symmetry of $v$ implies that
\[v(a_1, a_2, \dots, a_k) = (-1)^{|a_1||a_1|+n}v(a_2, a_1, \dots, a_k).\]
We define the Schouten bracket of $v\in\Pol(A, n-1)^k$ and $w\in\Pol(A, n-1)^l$ to be
\begin{align*}
[v, w](a_1,\dots,a_{k+l-1}) = &\sum_{\sigma\in S_{l, k-1}} \sgn(\sigma)^n(-1)^{\epsilon+\epsilon_1} v(w(a_{\sigma(1)},\dots,a_{\sigma(l)}), a_{\sigma(l+1)},\dots,a_{\sigma(k+l-1)}) \\
- &\sum_{\sigma\in S_{k, l-1}} \sgn(\sigma)^n(-1)^{\epsilon+\epsilon_2} w(v(a_{\sigma(1)},\dots,a_{\sigma(k)}), a_{\sigma(k+1)},\dots,a_{\sigma(k+l-1)}),
\end{align*}
where $(-1)^{\epsilon}$ denotes the sign coming from the Koszul sign rule applied to the permutation $\sigma$ of $a_i$ and the signs $\epsilon_i$ are
\begin{align*}
\epsilon_1 &= (|w|+l)(k+1)n + |v|n \\
\epsilon_2 &= (|v|-kn)(|w|-ln)+n(k+1)(|w|+1) + |v|n.
\end{align*}

The product of polyvector fields is defined to be
\[(v\cdot w)(a_1,\dots, a_{k+l}) = \sum_{\sigma\in S_{k.l}} \sgn(\sigma)^n(-1)^{\epsilon+\epsilon_1}v(a_{\sigma(1)},\dots,a_{\sigma(k)}) w(a_{\sigma(k+1)},\dots,a_{\sigma(k+l)}),\]
where the sign is
\[\epsilon_1 = |w|kn + \sum_{i=1}^k |a_{\sigma(i)}|(nl + |w|).\]

\subsection{Algebras}

Let us begin with the basic object in this section which is a weak (and shifted) version of Poisson algebras.

\begin{defn}
A \textit{$\widehat{\bP}_n$-algebra} is a cdga $A$ together with an $L_\infty$-algebra structure of degree $1-n$ such that the $L_\infty$ operations $l_k$ are polyderivations with respect to the multiplication. More explicitly, $l_k$ are multilinear operations of degree $1-(k-1)n$ satisfying the following equations:
\begin{itemize}
\item (Symmetry).
\[l_k(a_1, \dots, a_i, a_{i+1}, \dots, a_k) = (-1)^{|a_i||a_{i+1}| + n}l_k(a_1, \dots, a_{i+1}, a_i, \dots, a_k).\]

\item (Leibniz rule).
\[l_k(a_1, \dots, a_ka_{k+1}) = l_k(a_1, \dots, a_k)a_{k+1} + (-1)^{|a_k||a_{k+1}|} l_k(a_1, \dots, a_{k+1}) a_k.\]

\item (Jacobi identity).
\[0 = \sum_{k=1}^m (-1)^{nk(m-k)} \sum_{\sigma\in S_{k, m-k}} \sgn(\sigma)^n(-1)^{\epsilon} l_{m-k+1}(l_k(a_{\sigma(1)}, \dots, a_{\sigma(k)}), a_{\sigma(k+1)}, \dots, a_{\sigma(m)}),\]
where $\epsilon$ is the sign coming from the Koszul sign rule.
\end{itemize}
\end{defn}

Given a $\widehat{\bP}_n$-algebra $A$, the opposite algebra $A^{\op}$ is defined to be the same cdga together with operations $l_k^{\op} = (-1)^{k+1}l_k$.

There is also a strict version of Poisson algebras as follows.
\begin{defn}
A \textit{$\bP_n$-algebra} is a $\widehat{\bP}_n$-algebra such that the operations $l_k$ vanish for $k > 2$. In this case we denote the operation $l_2$ by $\{a,b\}$.
\end{defn}

\begin{defn}
A \textit{morphism of $\widehat{\bP}_n$-algebras} $f\colon A\rightarrow B$ is a chain map of complexes $f\colon A\rightarrow B$ strictly preserving the multiplication and the $L_\infty$ operations $l_k$.
\end{defn}

Here is an important example of a $\bP_{n+1}$-algebra. Observe that the Schouten bracket on $\Pol(A, n-1)$ has cohomological degree $-n$.

\begin{prop}
Let $A$ be a cdga. The product and Schouten bracket define a $\bP_{n+1}$-structure on the complex of $(n-1)$-shifted polyvector fields $\Pol(A, n-1)$.
\label{prop:Pnpolyvector}
\end{prop}

A $\bP_n$-structure on a cdga $A$ is given by a bivector $\pi_A\in \Pol(A, n-1)$ of degree $n+1$, so that
\[\{a, b\}:=\pi_A(a,b).\]
The Jacobi identity for the bracket then becomes
\[[\pi_A, \pi_A] = 0.\]

Given a $\bP_n$-algebra $A$, we can naturally produce a $\bP_{n+1}$-algebra $\Z(A)$ as follows.

\begin{defn}
Let $A$ be a $\bP_n$-algebra. Its \textit{Poisson center} is the $\bP_{n+1}$-algebra given by the completion
\[\Z(A) = \widehat{\Pol}(A, n-1)\]
of the algebra of $(n-1)$-shifted polyvector fields with respect to the weight grading. The Lie bracket is given by the Schouten bracket. The differential has two components: the differential on the module of K\"{a}hler differentials and $[\pi_A, -]$.
\end{defn}

\begin{remark}
Suppose $A$ is a non-dg Poisson algebra. Then $\Z(A)$ coincides with the Lichnerowicz--Poisson complex $C^\bullet_{LP}(A, A)$, see \cite[Section 1.4.8]{Fre}, whose zeroth cohomology is the space of Casimir functions. See also \cite[Theorem 2]{CW} for a relation between $\Z(A)$ and a Poisson analog of the Hochschild complex.
\end{remark}

\begin{remark}
We believe that if $A$ is cofibrant as a commutative dg algebra, $\Z(A)$ is a model of the center of $A\in\Alg_{\bP_n}$ in the sense of \cite[Definition 5.3.1.6]{HA}. We will return to this comparison in a future work.
\label{remark:centralizer}
\end{remark}

We have a morphism
\[\Z(A)\rightarrow A\]
of commutative dg algebras given by projecting to the weight zero part of polyvector fields.

\subsection{Modules}

Let $A$ be a $\bP_{n+1}$-algebra and $M$ a cdga.

\begin{defn}
A \textit{coisotropic structure} on a morphism of commutative dg algebras $f\colon A\rightarrow M$ is a $\bP_n$-algebra structure on $M$ and a lift
\[
\xymatrix{
A \ar@{-->}^{\tilde{f}}[r] \ar^{f}[dr] & \Z(M) \ar[d] \\
& M,
}
\]
where $\tilde{f}\colon A\rightarrow \Z(M)$ is a morphism of $\bP_{n+1}$-algebras.
\label{defn:coisotropic}
\end{defn}

Here is a way to unpack this definition. A coisotropic structure consists of maps \[f_k\colon A\rightarrow \Hom_M(\Sym^k(\Omega^1_M[n]), M)\] for $k\geq 0$, where $f_0=f$ is the original morphism. We define the maps \[f_k\colon A\otimes M^{\otimes k}\rightarrow M[-nk]\] by
\[f_k(a;m_1,\dots,m_k):= f_k(a)(m_1,\dots,m_k).\]
They satisfy the following equations:

\begin{itemize}
\item (Symmetry).
\begin{equation}
f_k(a;m_1, \dots, m_i, m_{i+1}, \dots, m_k) = (-1)^{|m_i||m_{i+1}|+n} f_k(a; m_1, \dots, m_{i+1}, m_i, \dots, m_k)
\label{eq:coisotropic1}
\end{equation}
for every $a\in A$ and $m_i\in M$.

\item (Derivation).

\begin{equation}
f_k(a; m_1, \dots, m_km_{k+1}) = f_k(a; m_1, \dots, m_k) m_{k+1} + (-1)^{|m_k||m_{k+1}|} f_k(a; m_1,\dots, m_{k+1}) m_k
\label{eq:coisotropic2}
\end{equation}
for every $a\in A$ and $m_i\in M$.

\item (Compatibility with the differential).

\begin{align}
\label{eq:coisotropic3}
&\d f_k(a; m_1, \dots, m_k) = \\
&\qquad f_k(\d a; m_1, \dots, m_k) + \sum_{i=1}^k (-1)^{|a|+\sum_{j=1}^{i-1}|m_j|+nk} f_k(a; m_1, \dots, \d m_i, \dots, m_k) \nonumber \\
&-\sum_{i=1}^k (-1)^{n(|a|+i-1)+|m_i| \sum_{j=i+1}^k|m_j|}\{f_{k-1}(a; m_1, \dots, \hat{m}_i, \dots, m_k), m_i\} \nonumber \\
&+ \sum_{i<j} (-1)^{|m_i|\sum_{l=1}^{i-1}|m_l| + |m_j|\sum_{l=1,l\neq i}^{j-1}|m_l| + n(i+j)+|a|} f_{k-1}(a; \{m_i, m_j\}, m_1, \dots, \hat{m_i},\dots,\hat{m_j},\dots m_k)\nonumber
\end{align}
for every $a \in A$ and $m_i\in M$.

\item (Compatibility with the brackets).

For every $a_1,a_2\in A$ and $m_i\in M$ we have
\begin{align}
&f_k(\{a_1, a_2\}; m_1,\dots, m_k) \nonumber \\
&\qquad = \sum_{i+j=k+1}\sum_{\sigma\in S_{j, i-1}} \sgn(\sigma)^n(-1)^{\epsilon+\epsilon_1} f_i(a_1; f_j(a_2; m_{\sigma(1)}, \dots, m_{\sigma(j)}), m_{\sigma(j+1)}, \dots, m_{\sigma(k)}) \nonumber \\
&\qquad- \sum_{i+j=k+1} \sum_{\sigma\in S_{j, i-1}} \sgn(\sigma)^n(-1)^{\epsilon+\epsilon_2} f_i(a_2; f_j(a_1; m_{\sigma(1)}, \dots, m_{\sigma(j)}), m_{\sigma(j+1)}, \dots, m_{\sigma(k)}),
\label{eq:coisotropic4}
\end{align}
where the signs are
\begin{align*}
\epsilon_1 &= (|a_2|+j)(i+1)n + |a_1|n \\
\epsilon_2 &= (|a_1|-jn)(|a_2|-in)+n(j+1)(|a_2|+1) + |a_1|n.
\end{align*}

\item (Compatibility with the product).

For every $a_1,a_2\in A$ and $m_i\in M$ we have
\begin{align}
\label{eq:coisotropic5}
&f_k(a_1a_2; m_1,\dots,m_k) \\
&\qquad= \sum_{i+j=k} \sum_{\sigma\in S_{i,j}} \sgn(\sigma)^n (-1)^{\epsilon+\epsilon_1} f_i(a_1;m_{\sigma(1)},\dots,m_{\sigma(i)})f_j(a_2;m_{\sigma(i+1)},\dots,m_{\sigma(k)}), \nonumber
\end{align}
where the sign is $\epsilon_1 = |a_2|ni+\sum_{l=1}^i |m_{\sigma(l)}|(nj+|a_2|)$.
\end{itemize}

\begin{remark}
Equation \eqref{eq:coisotropic4} for $k=0$ reads as
\[f_0(\{a_1, a_2\}) = (-1)^{|a_1|n}f_1(a_1; f_0(a_2)) - (-1)^{n(|a_2|+1) + |a_1||a_2|} f_1(a_2; f_0(a_1)).\]

In particular, the kernel of $f_0$ is closed under the Poisson bracket and so $\Spec M\rightarrow \Spec A$ is a coisotropic subscheme in the usual sense.
\end{remark}

\begin{example}
Several examples of coisotropic structures as above are constructed in \cite[Examples 3.20 and 3.21]{JS} from shifted Lagrangian structures.
\end{example}

\begin{remark}
The above definition can be made into a two-colored operad $\bP_{[n+1, n]}$ so that a $\bP_{[n+1, n]}$-algebra is given by a triple of a $\bP_{n+1}$-algebra $A$, a $\bP_n$-algebra $B$ and a morphism of $\bP_{n+1}$-algebras $A\rightarrow \Z(B)$. This allows one to define an $\infty$-groupoid of coisotropic structures which is studied in \cite{MS16}. In particular, in \cite[Section 2.3]{MS17} Melani and the author show that arbitrary smooth coisotropic subschemes possess a coisotropic structure in this sense up to homotopy.
\end{remark}

\subsection{Koszul duality}
\label{sect:Koszul}

For a complex $A$ we denote by $\T_\bullet(A[1])$ the tensor coalgebra. As a complex,
\[\T_\bullet(A[1])\cong \bigoplus_{k=0}^\infty A^{\otimes k}[k].\]
We denote an element of $A^{\otimes k}$ by $[a_1| \dots|a_k]$ for $a_i\in A$. The canonical element in $A^{\otimes 0}$ is denoted by $[]$.

The coproduct is given by deconcatenation, i.e.
\[\Delta [a_1| \dots|a_k] = \sum_{i=0}^k [a_1|\dots|a_i]\otimes [a_{i+1}|\dots|a_k].\]

Let us denote by $\wedge$ the concatenation product:
\[[a_1|\dots|a_i]\wedge [a_{i+1}|\dots|a_k] = [a_1|\dots|a_k].\]

Note that the deconcatenation coproduct and concatenation product do not form a bialgebra structure.

If $A$ is a cdga, we can introduce the bar differential on $\T_\bullet(A[1])$ and a commutative multiplication given by shuffles. That is,
\begin{align*}
\d[a_1|\dots|a_k] = &\sum_{i=1}^k (-1)^{\sum_{q=1}^{i-1}|a_q| + i-1} [a_1|\dots|\d a_i|\dots|a_k] \\
+&\sum_{i=1}^{k-1} (-1)^{\sum_{q=1}^i |a_q|+i} [a_1|\dots|a_i a_{i+1}|\dots|a_k]
\end{align*}
and
\[[a_1|\dots|a_k]\cdot [a_{k+1}|\dots|a_{k+m}] = \sum_{\sigma\in S_{k, m}} (-1)^{\epsilon} [a_{\sigma(1)}|\dots|a_{\sigma(k+m)}],\]
where the sign $\epsilon$ is determined by assigning degrees $|a_i|-1$ to $a_i$. The element $[]\in \T_\bullet(A[1])$ is the unit for the shuffle product. We refer the reader to \cite[Section 1]{GJ} for a detailed explanations of all signs involved.

Now let $A$ be a $\bP_{n+1}$-algebra. Then we can define a Lie bracket on $\T_\bullet(A[1])$ by
\begin{align}
\label{eq:PnKoszulbracket}
&\{[a_1|\dots|a_k], [b_1|\dots|b_m]\} \\
= &\sum_{i,j} (-1)^{\epsilon + |a_i| + n + 1} ([a_1|\dots|a_{i-1}]\cdot [b_1|\dots|b_{j-1}])\wedge[\{a_i, b_j\}]\wedge ([a_{i+1}|\dots|a_k]\cdot [b_{j+1}|\dots|b_m]) \nonumber.
\end{align}
The sign $\epsilon$ is determined by the following rule: an element $b$ moving past $\{a.-\}$ produces a sign $(-1)^{(|b|+1)(|a|+n)}$. For instance,
\[\{[a], [b|c]\} = (-1)^{|a|+n+1}[\{a, b\}|c] + (-1)^{|b|(|a|+n)+1} [b|\{a, c\}].\]

\begin{remark}
The same Poisson bracket was previously introduced by Fresse \cite[Section 3]{Fre} under the name ``shuffle Poisson bracket''.
\end{remark}

\begin{defn}
A \textit{$\bP_n$-bialgebra} is a $\bP_n$-algebra $\tilde{A}$ together with a coassociative comultiplication $\tilde{A}\rightarrow \tilde{A}\otimes \tilde{A}$ which is a morphism of $\bP_n$-algebras.
\end{defn}

\begin{prop}
The differential, multiplication, comultiplication and bracket defined above endow $\T_\bullet(A[1])$ with a $\bP_n$-bialgebra structure.
\label{prop:PnKoszul}
\end{prop}
\begin{proof}
See \cite[Proposition 4.1]{GJ} for the proof that $\T_\bullet(A[1])$ is a commutative dg bialgebra. We just need to show that the bracket is compatible with the other operations.

Let us first show that the Lie bracket is compatible with the coproduct. We will omit some obvious signs arising from a permutation of $a$ and $b$.

\begin{align*}
\{\Delta[a_1|\dots|a_k]&, \Delta[b_1|\dots|b_m]\} \\
= &\sum_{i, j} \{[a_1|\dots|a_i]\otimes [a_{i+1}|\dots|a_k], [b_1|\dots|b_j]\otimes [b_{j+1}|\dots|b_m]\} \\
= &\sum_{i, j} (-1)^{\epsilon}\{[a_1|\dots|a_i], [b_1|\dots|b_j]\}\otimes ([a_{i+1}|\dots|a_k]\cdot [b_{j+1}|\dots|b_m]) \\
+&\sum_{i, j} (-1)^{\epsilon}([a_1|\dots|a_i]\cdot [b_1|\dots|b_j])\otimes \{[a_{i+1}|\dots|a_k], [b_{j+1}|\dots|b_m]\} \\
= &\sum_{i,j,p,q} (-1)^{\epsilon} ([a_1|\dots|a_{p-1}]\cdot [b_1|\dots|b_{q-1}])\wedge [\{a_p, b_q\}]\wedge ([a_{p+1}|\dots|a_i]\cdot [b_{q+1}|\dots|b_j])\\
&\otimes ([a_{i+1}|\dots|a_k]\cdot [b_{j+1}|\dots|b_m]) \\
+&\sum_{i,j,p,q} (-1)^{\epsilon}([a_1|\dots|a_i]\cdot [b_1|\dots|b_j]) \\
&\otimes ([a_{i+1}|\dots|a_{p-1}]\cdot[b_{j+1}|\dots|b_{q-1}])\wedge [\{a_p, b_q\}]\wedge ([a_{p+1}|\dots|a_k]\cdot [b_{q+1}|\dots|b_m]) \\
= &\Delta\{[a_1|\dots|a_k], [b_1|\dots|b_m]\}.
\end{align*}
In the last equality we have used that the tensor coalgebra with a shuffle product is a bialgebra.

The fact that the Lie bracket is symmetric is obvious from the graded commutativity of the shuffle product.

The Jacobi identity and the Leibniz rule are morphisms $f\colon \T_\bullet(A[1])^{\otimes 3}\rightarrow \T_\bullet(A[1])$ satisfying \[\Delta_{\T_\bullet(A[1])}\circ f = (f\otimes m + m\otimes f)\circ \Delta_{\T_\bullet(A[1])^{\otimes 3}},\]
where $m\colon\T_\bullet(A[1])^{\otimes 3}\rightarrow \T_\bullet(A[1])$ is the multiplication map.

These are uniquely determined by the projections
\[\T_\bullet(A[1])^{\otimes 3}\rightarrow \T_\bullet(A[1])\rightarrow A[1]\]
to cogenerators. Therefore, to check the relevant identities, we just need to see that the components landing in $A$ are all zero.

\begin{itemize}
\item (Jacobi identity). The Lie bracket has a component in $A$ only if both arguments are in $A$. Therefore, the Jacobi identity in $\T_\bullet(A[1])$ reduces to the Jacobi identity in $A$ itself.

\item (Leibniz rule). The Leibniz rule \[\{a, bc\} = \{a, b\}c + (-1)^{|b||c|}\{a, c\}b,\quad a,b,c\in \T_\bullet(A[1])\] has components in $A$ only if either $a$ or $b$ are 1. In that case the Leibniz rule is tautologically true.

\item (Compatibility with the differential). The compatibility relation \[\d\{a, b\} = (-1)^{n+1}\{\d a, b\} + (-1)^{|a|+n+1}\{a, \d b\}\] has components in $A$ if either both $a$ and $b$ are in $A$ or one of them is in $A$ and the other one is in $A^{\otimes 2}$. In the first case the compatibility of the bracket on $\T_\bullet(A[1])$ with the differential reduces to the compatibility of the bracket on $A$ with the differential. In the second case the $A$ component of the equation is
\[(-1)^{|b_1|}\{a, b_1\}b_2 + (-1)^{|b_1|(|a|+n+1)} b_1\{a, b_2\} = (-1)^{|b_1|}\{a, b_1b_2\}.\]

After multiplying through by $(-1)^{|b_1|}$ we get the Leibniz rule for the bracket on $A$.
\end{itemize}
\end{proof}

\begin{remark}
For $A$ a $\bP_{n+1}$-algebra the coalgebra $\T_\bullet(A[1])^{\cop}$ with the opposite coproduct is isomorphic to $\T_\bullet(A^{\op}[1])$ as a $\bP_n$-bialgebra via
\begin{equation}
[a_1|\dots|a_k]\mapsto (-1)^{k+\sum_{i<j} (|a_i|+1)(|a_j|+1)} [a_k|\dots|a_1].
\label{eq:cofreecop}
\end{equation}
\end{remark}

\subsection{Coisotropic intersection}

Let us now describe a relative version of the previous statement. Let $A$ be a $\bP_{n+1}$-algebra and $f\colon A\rightarrow M$ a coisotropic morphism. We are going to define a $\bP_n$-algebra structure on $\T_\bullet(A[1])\otimes M$, the one-sided bar complex of $M$. As before, we denote elements of $\T_\bullet(A[1])\otimes M$ by $[a_1|\dots|a_k|m]$.

Recall that the bar differential is given by
\begin{align*}
\d[a_1|\dots|a_k|m] = &\sum_{i=1}^k (-1)^{\sum_{q=1}^{i-1}|a_q|+i-1} [a_1|\dots|\d a_i|\dots|a_k|m]\\
+ &(-1)^{\sum_{q=1}^k|a_q|+k} [a_1|\dots|a_k|\d m] \\
+ &\sum_{i=1}^{k-1} (-1)^{\sum_{q=1}^i|a_q|+i} [a_1|\dots|a_i a_{i+1}|\dots|a_k|m] \\
+ &(-1)^{\sum_{q=1}^k|a_q|+k} [a_1|\dots|a_k m].
\end{align*}

One has an obvious coaction map making $\T_\bullet(A[1])\otimes M$ into a left dg $\T_\bullet(A[1])$-comodule. As a graded $\T_\bullet(A[1])$-comodule, $\T_\bullet(A[1])\otimes M$ is cofree.

Introduce a commutative multiplication on $\T_\bullet(A[1])\otimes M$ where the multiplication on $\T_\bullet(A[1])$ is given by shuffles as before and the multiplication on $M$ is coming from its cdga structure. The $L_\infty$ operations we are about to introduce are multiderivations, so by the relation
\[[a_1|\dots|a_k|m] = [a_1|\dots|a_k|1]\cdot[m]\]
it is enough to specify them when the arguments are either in $\T_\bullet(A[1])$ or in $M$. If all arguments are in $\T_\bullet(A[1])$, we define the brackets as before. We let
\begin{equation}
l_{k+1}([a_1|\dots|a_p|1], [m_1], \dots, [m_k]) =  (-1)^{(\sum_{q=1}^p |a_q|+p)(1-nk)} [a_1|\dots|a_{p-1}|f_k(a_p; m_1,\dots, m_k)]
\label{eq:barcomplexlk}
\end{equation}
and
\begin{equation}
l_2([m_1], [m_2]) = [\{m_1, m_2\}],
\label{eq:barcomplexl2}
\end{equation}
where the Poisson bracket on the right is the bracket in $M$. All the other brackets are defined to be zero.

\begin{defn}
A \textit{left $\widehat{\bP}_n$-comodule} $\tilde{M}$ over a $\bP_n$-bialgebra $\tilde{A}$ is a $\widehat{\bP}_n$-algebra $\tilde{M}$ together with a coassociative left coaction map $\tilde{M}\rightarrow \tilde{A}\otimes \tilde{M}$ which is a morphism of $\widehat{\bP}_n$-algebras.
\end{defn}

\begin{prop}
The differential, coaction, multiplication and $L_\infty$ operations defined above make $\T_\bullet(A[1])\otimes M$ into a left $\widehat{\bP}_n$-comodule over $\T_\bullet(A[1])$.
\end{prop}
\begin{proof}
To prove compatibility of the $L_\infty$ operations with the coaction, it is enough to assume each argument is either in $M$ or in $\T_\bullet(A[1])$. If all arguments are in $\T_\bullet(A[1])$, the compatibility with the coaction was checked in Proposition \ref{prop:PnKoszul}. If all arguments are in $M$ and $k=2$ we have
\[\Delta l_2([m_1], [m_2]) = []\otimes [\{m_1, m_2\}]\]
and
\[l_2(\Delta([m_1]), \Delta([m_2])) = l_2([]\otimes [m_1], []\otimes [m_2]) = []\otimes [\{m_1, m_2\}].\]

If all but one arguments are in $M$ and $k$ is arbitrary we have
\begin{align*}
&l_k(\Delta[a_1|\dots|a_p|1], []\otimes [m_1], \dots, []\otimes [m_{k-1}])\\
&\qquad=\sum_{i=0}^p l_k([a_1|\dots|a_i]\otimes [a_{i+1}|\dots|a_p|1], []\otimes [m_1], \dots, []\otimes [m_{k-1}]) \\
&\qquad=\sum_{i=0}^p (-1)^{\sum_{q=1}^i |a_q|(1-(k-1)n)} [a_1|\dots|a_i]\otimes l_k([a_{i+1}|\dots|a_p|1], [m_1], \dots, [m_{k-1}]) \\
&\qquad=\sum_{i=0}^p (-1)^{\sum_{q=1}^p |a_q|(1-(k-1)n)} [a_1|\dots|a_i]\otimes [a_{i+1}|\dots|f_{k-1}(a_p;m_1,\dots, m_{k-1})]
\end{align*}
and
\begin{align*}
&\Delta l_k([a_1|\dots|a_p|1], [m_1], \dots, [m_{k-1}]) \\
&\qquad= (-1)^{\sum_{q=1}^p |a_q|(1-(k-1)n)} \Delta [a_1|\dots|a_{p-1}|f_{k-1}(a_p;m_1,\dots,m_{k-1})]\\
&\qquad= (-1)^{\sum_{q=1}^p |a_q|(1-(k-1)n)}\sum_{i=0}^p [a_1|\dots|a_i]\otimes[a_{i+1}|\dots|f_{k-1}(a_p;m_1,\dots, m_{k-1})].
\end{align*}

Therefore, as before it is enough to check symmetry, the Leibniz rule and Jacobi identity only after projecting to $M$. The operation $l_k$ has a component in $M$ if either all but one arguments are in $M$ and one argument is in $A$ or $k=2$ and both arguments are in $M$.

\begin{itemize}
\item (Symmetry). Symmetry is clear for $l_2(m_1, m_2)$. For $l_k(a, m_1, \dots, m_{k-1})$ symmetry in the $m_i$ variables follows from the symmetry property \eqref{eq:coisotropic1} of $f_{k-1}$.

\item (Leibniz rule). If $k=2$ we need to check that
\[l_2([m_1], [m_2m_3]) = l_2([m_1], [m_2])[m_3] + (-1)^{|m_2||m_3|} l_2([m_1], [m_3])[m_2].\]
This is just an expression for the Leibniz rule in $M$. For any $k$ we also need to check that
\begin{align*}
l_k([a|1], [m_1], \dots, [m_{k-1}m_k]) &= l_k([a|1], [m_1], \dots, [m_{k-1}])[m_k] \\
&+ (-1)^{|m_{k-1}||m_k|} l_k([a|1], [m_1], \dots, [m_k])[m_{k-1}].
\end{align*}
This immediately follows from the derivation property \eqref{eq:coisotropic2} of $f_{k-1}$

\item (Jacobi identity).

The Jacobi identity has a component in $M$ in the following four cases:
\begin{enumerate}
\item All arguments are in $M$. In this case we get the Jacobi identity for the bracket in $M$.

\item One argument is in $A$, the rest are in $M$.

The Jacobi identity is
\begin{align*}
0=&(-1)^{nk}l_{k+1}([\d a|1], [m_1], \dots, [m_k])\\
&+ \d l_{k+1}([a|1], [m_1], \dots, [m_k]) \\
&+\sum_i (-1)^{|a|+\sum_{j=1}^{i-1} |m_j| + nk+1}l_{k+1}([a|1], [m_1], \dots, [\d m_i], \dots, [m_k]) \\
&+\sum_{i < j} (-1)^{\epsilon'}l_k([a|1], \{[m_i], [m_j]\}, \dots) \\
&+ \sum_i (-1)^{|m_i|\sum_{j=i+1}^k |m_j| + in}\{l_k([a|1], [m_1], \dots, \widehat{[m_i]}, \dots, [m_k]), [m_i]\},
\end{align*}
where the sign is
\[\epsilon'=|m_i|\sum_{p=1}^{i-1} |m_p| + |m_j|\sum_{p=1,p\neq i}^{j-1} |m_p| + n(i+j) + (|a|+1)(1-n).\]

Substituting $l_k$ in terms of $f_{k-1}$ from equation \eqref{eq:barcomplexlk} we obtain
\begin{align*}
0=&(-1)^{nk}(-1)^{|a|(1-nk)} f_k(\d a; m_1, \dots, m_k)\\
&+ (-1)^{(|a|+1)(1-nk)} \d f_k(a; m_1, \dots, m_k) \\
&+\sum_i (-1)^{|a|+\sum_{j=1}^{i-1} |m_j| + nk + (|a|+1)(1-nk)}f_k(a; m_1, \dots, \d m_i, \dots, m_k) \\
&+\sum_{i < j} (-1)^{|m_i|\sum_{p=1}^{i-1} |m_p| + |m_j|\sum_{p=1,p\neq i}^{j-1} |m_p| + n(i+j) + (|a|+1)nk}f_{k-1}(a; \{m_i, m_j\}, \dots) \\
&+ \sum_i (-1)^{|m_i|\sum_{j=i+1}^k |m_j| + in + (|a|+1)(1-n(k-1))}\{f_{k-1}(a; m_1, \dots, \hat{m_i}, \dots, m_k), m_i\}.
\end{align*}

After clearing out the signs, the equation coincides with \eqref{eq:coisotropic3}.

\item Two arguments are in $A$, the rest are in $M$.

The Jacobi identity is
\begin{align*}
0 &= (-1)^{|a_1|+n+1}l_{k+1}([\{a_1, a_2\}|1], [m_1], \dots, [m_k]) \\
&+\sum_{i+j = k+1} (-1)^{n(j+1)(k-j-1)} \sum_{\sigma\in S_{j, k-j}} \sgn(\sigma)^n (-1)^{\epsilon_m} (-1)^{nj + (|a_1|+1)(1+nj)}\times \\
&\qquad l_{i+1}([a_1|1], l_{j+1}([a_2|1], [m_{\sigma(1)}], \dots, [m_{\sigma(j)}]), [m_{\sigma(j+1)}], \dots, [m_{\sigma(k)}]) \\
&+ \sum_{i+j = k+1} (-1)^{n(j+1)(k-j-1)} \sum_{\sigma\in S_{j, k-j}} \sgn(\sigma)^n (-1)^{\epsilon_m} (-1)^{n(j+1) + (|a_2|+1)(nj + |a_1|)}\times \\
&\qquad l_{i+1}([a_2|1], l_{j+1}([a_1|1], [m_{\sigma(1)}], \dots, [m_{\sigma(j)}]), [m_{\sigma(j+1)}], \dots, [m_{\sigma(k)}]).
\end{align*}

Substituting $l_k$ in terms of $f_{k-1}$ from equation \eqref{eq:barcomplexlk} we obtain

\begin{align*}
0 &= (-1)^{|a_1|+n+1}(-1)^{(|a_1|+|a_2|-n+1)(1-nk)}f_k(\{a_1, a_2\}; m_1, \dots, m_k) \\
&+\sum_{i+j = k+1} (-1)^{n(j+1)(k-j-1)} \sum_{\sigma\in S_{j, k-j}} \sgn(\sigma)^n (-1)^{\epsilon_m+(|a_1|+1)n(k+1) + (|a_2|+1)(1+nj)+nj}\times \\
&\qquad f_i(a_1; f_j(a_2; m_{\sigma(1)}, \dots, m_{\sigma(j)}), m_{\sigma(j+1)}, \dots, m_{\sigma(k)}) \\
&+ \sum_{i+j = k+1} (-1)^{n(j+1)(k-j-1)} \sum_{\sigma\in S_{j, k-j}} \sgn(\sigma)^n (-1)^{} \times \\
&\qquad f_i(a_2; f_j(a_1; m_{\sigma(1)}, \dots, m_{\sigma(j)}), m_{\sigma(j+1)}, \dots, m_{\sigma(k)}),
\end{align*}
where the last sign is
\[\epsilon'=\epsilon_m+(|a_2|+1)(1-n(k+1)+|a_1|) + (|a_1|+1)(1-nj)+n(j+1).\]

After rearranging the signs, we get \eqref{eq:coisotropic4}.

\item One argument is in $(A[1])^{\otimes 2}$, the rest are in $M$.

The Jacobi identity is
\begin{align*}
0 &= (-1)^{nk}l_{k+1}(\d[a_1|a_2|1], [m_1], \dots, [m_k]) + \d l_{k+1}([a_1|a_2|1], [m_1], \dots, [m_k]) \\
&+\sum_{\substack{i+j=k\\i,j>0}} \sum_{\sigma \in S_{j, i}} \sgn(\sigma)^n (-1)^{nk(j+1) + \epsilon} \times \\
&\qquad l_{i+1}(l_{j+1}([a_1|a_2|1], [m_{\sigma(1)}], \dots, [m_{\sigma(j)}]), [m_{\sigma(j+1)}], \dots, [m_{\sigma(k)}]).
\end{align*}

The projection of each term to $M$ is
\begin{align*}
&l_{k+1}(\d[a_1|a_2|1], [m_1], \dots, [m_k])\\
&\qquad= (-1)^{|a_1|+1}l_{k+1}([a_1 a_2|1], [m_1], \dots, [m_k]) \\
&\qquad+ (-1)^{|a_1|+|a_2|} l_{k+1}([a_1|f_0(a_2)], [m_1], \dots, [m_k]) \\
&\qquad= (-1)^{|a_1|+1 + (|a_1|+|a_2|+1)(1-nk)} f_k(a_1 a_2; m_1, \dots, m_k) \\
&\qquad+ (-1)^{|a_1|+|a_2|(1+\sum_{i=1}^k |m_i|) + (|a_1|+1)(1-nk)} f_k(a_1; m_1, \dots, m_k) f_0(a_2),\\
&\d l_{k+1}([a_1|a_2|1], [m_1], \dots, [m_k]) = (-1)^{|a_1| + 1 + (|a_1|+|a_2|)(1-nk)} f_0(a_1)f_k(a_2; m_1, \dots, m_k),\\
&l_{i+1}(l_{j+1}([a_1|a_2|1], [m_{\sigma(1)}], \dots, [m_{\sigma(j)}]), [m_{\sigma(j+1)}], \dots, [m_{\sigma(k)}]) \\
&\qquad=(-1)^{(|a_1|+|a_2|)(1-nj)} l_{i+1}([a_1|f_j(a_2; m_{\sigma(1)}, \dots, m_{\sigma(j)})], [m_{\sigma(j+1)}], \dots, [m_{\sigma(k)}]) \\
&\qquad=(-1)^{(|a_1|+|a_2|)(1-nj) + (|a_2|+\sum_{i=1}^j |m_{\sigma(i)}| + nj)\sum_{p=1}^{k-j} |m_{\sigma(j+p)}| + (|a_1|+1)(1-ni)}\times \\
&\qquad\qquad f_i(a_1; m_{\sigma(j+1)}, \dots, m_{\sigma(k)}) f_j(a_2; m_{\sigma(1)}, \dots, m_{\sigma(j)})
\end{align*}

Let us denote by $\overline{\sigma}\in S_{i, j}$ the shuffle obtained from $\sigma$ by swapping the blocks $\sigma(1),\dots,\sigma(j)$ and $\sigma(j+1),\dots, \sigma(k)$. That is, $\overline{\sigma}(p) = \sigma(j+p)$ for $1\leq p\leq i$ and $\overline{\sigma}(p) = \sigma(p-i)$ for $i<p\leq k$. Denote by $\overline{\epsilon}$ the Koszul sign corresponding to the shuffle $\overline{\sigma}$. We have
\[\sgn(\overline{\sigma}) = \sgn(\sigma) (-1)^{j(k-j)}\]
and
\[(-1)^{\overline{\epsilon}} = (-1)^{\epsilon} (-1)^{\sum_{i=1}^j |m_{\sigma(i)}| \sum_{p=1}^{k-j} |m_{\sigma(j+p)}|}.\]

The Jacobi identity becomes
\begin{align*}
0 &= (-1)^{|a_1|+(|a_1|+|a_2|)(1-nk)} f_k(a_1 a_2; m_1, \dots, m_k)\\
&-(-1)^{|a_1|+|a_2|(1+\sum_{i=1}^k |m_i|) + |a_1|(1-nk)} f_k(a_1; m_1, \dots, m_k) f_0(a_2) \\
&-(-1)^{|a_1|+(|a_1|+|a_2|)(1-nk)} f_0(a_1)f_k(a_2; m_1, \dots, m_k) \\
&+\sum_{i+j=k;i,j>0}\sum_{\overline{\sigma}\in S_{i,j}}\sgn(\overline{\sigma})^n (-1)^{\epsilon'} f_i(a_1; m_{\sigma(j+1)}, \dots, m_{\sigma(k)}) f_j(a_2; m_{\sigma(1)}, \dots, m_{\sigma(j)}),
\end{align*}
where the sign is
\[\epsilon' = (|a_1|+|a_2|)(1-nj) + (|a_2| + nj)\sum_{p=1}^{k-j} |m_{\sigma(j+p)}| + (|a_1|+1)(1-ni) + n(j+k).\]

Rearranging the signs, we obtain \eqref{eq:coisotropic5}.
\end{enumerate}
\end{itemize}
\end{proof}

In the same way we can make $M\otimes \T_\bullet(A[1])$ into a $\widehat{\bP}_n$-algebra compatibly with the right coaction of $\T_\bullet(A[1])$. The bar differential on $M\otimes \T_\bullet(A[1])$ is given by
\begin{align*}
\d[m|a_1|\dots|a_n] = &[\d m|a_1|\dots|a_n] \\
+ &\sum_{i=1}^n (-1)^{\sum_{q=1}^{i-1}|a_q|+i-1+|m|} [m|a_1|\dots|\d a_i|\dots|a_n] \\
+ &(-1)^{|m|+|a_1|+1}[m a_1|\dots|a_n] \\
+ &\sum_{i=1}^{n-1} (-1)^{\sum_{q=1}^i |a_q|+i+|m|} [m|a_1|\dots|a_i a_{i+1}|\dots|a_n].
\end{align*}

Moreover, $M\otimes \T_\bullet(A[1])$ is isomorphic to $\T_\bullet(A^{\op}[1])^{\cop}\otimes M^{\op}$ as right $\T_\bullet(A[1])$-comodules using the isomorphism \eqref{eq:cofreecop}. Here $M^{\op}$ represents the same cdga with the opposite bracket and the coisotropic structure given by $f_k^{\op} = (-1)^k f_k$. Using the previous theorem, we can make $M\otimes \T_\bullet(A[1])$ into a right $\widehat{\bP}_n$-comodule over $\T_\bullet(A[1])$.

Let us now combine left and right comodules.

\begin{thm}
Let $A$ be a $\bP_{n+1}$-algebra and $A\rightarrow M$ and $A\rightarrow N$ two coisotropic morphisms. Then the two-sided bar complex $N\otimes^{\bL}_A M$ has a natural structure of a $\widehat{\bP}_n$-algebra such that the natural projection $N^{\op}\otimes M\rightarrow N\otimes_A^{\bL} M$ is morphism of $\widehat{\bP}_n$-algebras.
\label{thm:coisotropicintersection}
\end{thm}
\begin{proof}
Let $\tilde{A} = \T_\bullet(A[1])$, $\tilde{N} = N\otimes \tilde{A}$ and $\tilde{M}=\tilde{A}\otimes M$. Then $\tilde{A}$ is a $\bP_n$-bialgebra, $\tilde{N}$ a right $\widehat{\bP}_n$-comodule and $\tilde{M}$ a left $\widehat{\bP}_n$-comodule over $\tilde{A}$.

We will first show that the cotensor product $\tilde{N}\otimes^{\tilde{A}} \tilde{M}$ is closed under the $\widehat{\bP}_n$-structures coming from $\tilde{N}\otimes\tilde{M}$.

Recall that
\[\tilde{N}\otimes^{\tilde{A}} \tilde{M}:=\eq(\tilde{N}\otimes \tilde{M}\rightrightarrows \tilde{N}\otimes \tilde{A}\otimes \tilde{M}),\]
where the two maps are coactions on $\tilde{M}$ and $\tilde{N}$ and the equalizer is the strict equalizer in the category of complexes. By definition the coaction
\[\tilde{M}\stackrel{\Delta_M}\rightarrow \tilde{A}\otimes \tilde{M}\]
is a morphism of $\widehat{\bP}_n$-algebras, so
\[\tilde{N}\otimes \tilde{M}\stackrel{\id_{\tilde{N}}\otimes \Delta_M}\rightarrow \tilde{N}\otimes \tilde{A}\otimes \tilde{M}\]
is also a morphism of $\widehat{\bP}_n$-algebras, but the forgetful functor from $\widehat{\bP}_n$-algebras to complexes creates limits, so the equalizer is also a $\widehat{\bP}_n$-algebra.

To conclude the proof of the theorem, we are going to construct an isomorphism
\[\tilde{N}\otimes^A \tilde{M}\cong N\otimes^\bL_A M.\]

The coproduct $\Delta\colon \tilde{A}\rightarrow \tilde{A}\otimes \tilde{A}$ induces an isomorphism
\[\Delta\colon \tilde{A}\rightarrow \eq(\tilde{A}\otimes \tilde{A}\rightrightarrows \tilde{A}\otimes \tilde{A}\otimes\tilde{A}),\]
where the two maps are $\Delta\otimes\id$ and $\id\otimes\Delta$. Therefore,
\[N\otimes^{\bL}_A M= N\otimes \T_\bullet(A[1])\otimes M\stackrel{\id_N\otimes \Delta\otimes \id_M}\longrightarrow N\otimes \T_\bullet(A[1])\otimes \T_\bullet(A[1])\otimes M\]
induces an isomorphism $N\otimes^\bL_A M\xrightarrow{\sim} \tilde{N}\otimes^A \tilde{M}$.
\end{proof}

\begin{remark}
Suppose $A\rightarrow M$ and $A\rightarrow N$ are two coisotropic morphisms as in the previous Theorem. Then any model of their derived intersection is quasi-isomorphic to the two-sided bar construction $N\otimes^\bL_A M$ and hence by the homotopy transfer theorem \cite[Section 10.3]{LV} we get an induced homotopy $\bP_n$-structure on the given model.
\end{remark}

\section{Classical Hamiltonian reduction}

Let $\g$ be a finite-dimensional dg Lie algebra over $k$ concentrated in non-positive degrees. In this section we apply results of the previous section to the $\bP_2$-algebra $A=\C^\bullet(\g, \Sym\g)$. The results of this section generalize in a straightforward way to $n$-shifted Hamiltonian reduction in which case we replace $A$ by the $\bP_{n+2}$-algebra $\C^\bullet(\g, \Sym(\g[-n]))$.

\subsection{Chevalley-Eilenberg complex}
Let $V$ be a $\g$-representation. The Chevalley--Eilenberg complex $\C^\bullet(\g, V)$ is defined to be
\[\C^\bullet(\g, V) = \Hom(\Sym(\g[1]), V)\]
with the differential
\begin{align}
(\d f)(x_1,\dots, x_n) &= \d f(x_1, \dots, x_n) \nonumber \\
&+\sum_{i=1}^n (-1)^{\sum_{p=1}^{i-1}|x_p| + |f| + n+1} f(x_1, \dots, \d x_i, \dots, x_n) \nonumber \\
&+\sum_{i<j} (-1)^{|x_i|\sum_{p=1}^{i-1} |x_p| + |x_j|\sum_{p=1,p\neq i}^{j-1} |x_p| + i + j + |f|} f([x_i, x_j], x_1, \dots, \widehat{x}_i, \dots, \widehat{x}_j, \dots, x_n) \nonumber \\
&+\sum_i (-1)^{|x_i|(\sum_{p=1}^{i-1} |x_p| + |f| + n + 1) + |f|+i+1} x_i f(x_1, \dots, \widehat{x}_i, \dots, x_n).
\label{eq:CEdifferential}
\end{align}

Here $|f|$ is the degree of $f$ in $\Hom(\Sym(\g[1]), V)$ and we have used the d\'{e}calage isomorphism as in \eqref{eq:decalage} to identify $\Hom(\Sym(\g[1]), -)$ with antisymmetric functions on $\g$.

The product
\begin{equation}
\smile\colon\C^\bullet(\g, A)\otimes \C^\bullet(\g, B)\rightarrow \C^\bullet(\g, A\otimes B)
\label{eq:CEproduct}
\end{equation}
is defined to be
\[(v\smile w)(x_1,\dots, x_{k+l}) = \sum_{\sigma\in S_{k.l}} \sgn(\sigma)(-1)^{\epsilon+\epsilon_1}v(x_{\sigma(1)},\dots,x_{\sigma(k)})\otimes w(x_{\sigma(k+1)},\dots,x_{\sigma(k+l)}),\]
where the sign is
\[\epsilon_1 = |w|k + \sum_{i=1}^k |x_{\sigma(i)}|(l + |w|).\]

\begin{remark}
Due to our finiteness assumptions on $\g$, we have an isomorphism
\[\C^\bullet(\g, V)\cong \Sym(\g^*[-1])\otimes V.\]

In particular, if $V$ is a semi-free commutative algebra, so is $\C^\bullet(\g, V)$.
\end{remark}

The algebra $\Sym\g$ has the Kirillov--Kostant Poisson structure given on the generators by $\pi(x_1, x_2) = [x_1, x_2]$ for $x_i\in\g$. The center of this $\bP_1$-algebra can be computed to be
\[\Z(\Sym \g) \cong \C^\bullet(\g, \Sym \g)\]
with the bracket
\begin{align*}
[v, w](x_1,\dots,x_{k+l-1}) = &\sum_{\sigma\in S_{l, k-1}} \sgn(\sigma)(-1)^{\epsilon+\epsilon_1} v(w(x_{\sigma(1)},\dots,x_{\sigma(l)}), x_{\sigma(l+1)},\dots,x_{\sigma(k+l-1)}) \\
- &\sum_{\sigma\in S_{k, l-1}} \sgn(\sigma)(-1)^{\epsilon+\epsilon_2} w(v(x_{\sigma(1)},\dots,x_{\sigma(k)}), x_{\sigma(k+1)},\dots,x_{\sigma(k+l-1)}),
\end{align*}
where $(-1)^{\epsilon}$ denotes the sign coming from the Koszul sign rule applied to the permutation $\sigma$ of $x_i$ and the signs $\epsilon_i$ are
\begin{align*}
\epsilon_1 &= (|w|+l)(k+1) + |v| \\
\epsilon_2 &= (|v|-k)(|w|-l)+(k+1)(|w|+1) + |v|.
\end{align*}

\subsection{Hamiltonian reduction}
\label{sect:classicalhamreduction}
Let $B$ be a $\bP_1$-algebra with a $\g$-action preserving the Poisson bracket. We denote by $a\colon \g\rightarrow \Der(B)$ the action map.

\begin{defn}
A $\g$-equivariant morphism of complexes $\mu\colon \g\rightarrow B$ is a \textit{moment map} for the $\g$-action on $B$ if the equation
\[\{\mu(x), b\} = a(x).b\]
is satisfied for all $x\in\g$ and $b\in B$. In this case we say that the $\g$-action is \textit{Hamiltonian}.
\label{def:momentmap}
\end{defn}

\begin{remark}
One can replace $\g$-equivariance in the definition of the moment map with the condition that the induced map $\Sym\g\rightarrow B$ is a morphism of $\bP_1$-algebras.
\end{remark}

\begin{defn}
Suppose $B$ is a $\bP_1$-algebra equipped with a $\g$-action and a moment map $\mu\colon \g\rightarrow B$. Its \textit{Hamiltonian reduction} is
\[B//\Sym\g:=\C^\bullet(\g, k)\otimes^{\bL}_{\C^\bullet(\g,\Sym\g)} \C^\bullet(\g, B).\]
\end{defn}

We will introduce a $\widehat{\bP}_1$-structure on this complex later in Corollary \ref{cor:BRSTcoisotropic}. Let us just mention a different complex used in derived Hamiltonian reduction called the classical BRST complex \cite{KS}
\[\C^\bullet(\g, \Sym(\g[1])\otimes B).\]
Here the differential on $\Sym(\g[1])\otimes B$ is the Koszul differential: given \[x_1\wedge \dots\wedge x_n\otimes b\in\Sym(\g[1])\otimes B\] we let
\begin{align*}
\d(x_1\wedge \dots\wedge x_n\otimes b) &= \sum_{i=1}^n (-1)^{(|x_i|+1)(\sum_{q=1}^{i-1} |x_q| + i - 1)} \d x_i\wedge x_1\wedge \dots\wedge \hat{x}_i\wedge \dots\wedge x_n\otimes b \\
&- \sum_{i=1}^n (-1)^{|x_i|\sum_{q=i+1}^n(|x_q|+1) + \sum_{q=1}^{i-1}(|x_q|+1)+|x_i|} x_1\wedge \dots\wedge \hat{x}_i\wedge \dots\wedge x_n\otimes \mu(x_i)b \\
&+ (-1)^{\sum_{q=1}^n |x_q| + n} x_1\wedge \dots\wedge x_n \otimes \d b.
\end{align*}

One can introduce a Poisson bracket on the classical BRST complex as follows. As a graded commutative algebra, the classical BRST complex is generated by $\g^*[-1]$, $\g[1]$ and $B$. We keep the bracket on $B$ and let the bracket between an element $\phi\in\g^*[-1]$ and an element $x\in\g[1]$ be the natural pairing: $\{\phi, x\}:= \phi(x)$. Then $\d$ is a derivation of the bracket precisely due to the moment map equation. In this way the classical BRST complex becomes a $\bP_1$-algebra.

\subsection{Hamiltonian reduction as a coisotropic intersection}

As a plain graded commutative algebra, $\C^\bullet(\g, B)\cong B\otimes \Sym(\g^*[-1])$, so its module of derivations is isomorphic to \[\T_B\otimes\Sym(\g^*[-1])\oplus B\otimes\g[1]\otimes\Sym(\g^*[-1])\] with the differential given by the sum of internal differentials on each term and the action map $\g\rightarrow \T_B$. Therefore, the Poisson center of $\C^\bullet(\g, B)$ is
\[\Z(\C^\bullet(\g, B))\cong \C^\bullet(\g, \widehat{\Sym}(\T_B[-1])\otimes \widehat{\Sym}(\g)).\]

Given a Hamiltonian $\g$-action on $B$, let us define the morphism
\[\C^\bullet(\g, \Sym\g)\rightarrow \Z(\C^\bullet(\g, B))\]
as follows. The cdga $\C^\bullet(\g, \Sym\g)$ is generated by $\C^\bullet(\g, k)$ and $\g\subset \Sym\g$. We let
\[\C^\bullet(\g, k)\hookrightarrow \C^\bullet(\g, \widehat{\Sym}(\T_B[-1])\otimes \widehat{\Sym}(\g))\]
be the natural embedding. The map
\[\g\rightarrow \C^\bullet(\g, \widehat{\Sym}(\T_B[-1])\otimes \widehat{\Sym}(\g))\]
is given by $x\mapsto \mu(x) - x$ for $v\in \g$.

\begin{prop}
Let $B$ be a $\bP_1$-algebra with a Hamiltonian $\g$-action. Then the morphism \[\C^\bullet(\g, \mu)\colon \C^\bullet(\g, \Sym \g)\rightarrow \C^\bullet(\g, B)\]
is coisotropic.
\label{prop:hamiltoniancoisotropic}
\end{prop}
\begin{proof}
It is enough to check that the morphism we have defined on generators commutes with the differential and the brackets.

Indeed, it is clear that the embedding $\C^\bullet(\g, k)\hookrightarrow \Z(\C^\bullet(\g, B))$ commutes with differentials. For $x\in \g$
\[\d \mu(x) + [\pi, \mu(x)] - \d x - (-1)^{|x|}a(x) = \d \mu(x) - \d x = \mu(\d x) - \d x,\]
where in the first equality we have used the moment map equation
\[[\pi, \mu(x)](b) = (-1)^{|x|}\{\mu(x), b\} = (-1)^{|x|}a(x).b.\]

It is also clear that the morphism commutes with brackets as $B$ Poisson-commutes with $\C^\bullet(\g, \Sym(\g))\hookrightarrow \Z(\C^\bullet(\g, B))$.
\end{proof}

\begin{example}
Let $B = k$ with the trivial $\g$-action and $\mu = 0$.

The morphism $\C^\bullet(\g, \Sym\g)\rightarrow \C^\bullet(\g, k)$ given by the counit $\Sym\g\rightarrow k$ possesses a coisotropic structure given by the composite of the antipode $S\colon \Sym\g\rightarrow\Sym\g$ with the completion map
\[\C^\bullet(\g, \Sym\g)\stackrel{S}\rightarrow \C^\bullet(\g, \Sym\g)\rightarrow \Z(\C^\bullet(\g, k))\cong \C^\bullet(\g, \widehat{\Sym}(\g)).\]
\end{example}

\begin{cor}
The Poisson reduction
\[B//\Sym\g=\C^\bullet(\g, k)\otimes_{\C^\bullet(\g, \Sym \g)}^{\bL} \C^\bullet(\g, B)\]
carries a natural $\widehat{\bP}_1$-structure. Moreover, there is a zig-zag of quasi-isomorphisms of cdgas between $B//\Sym\g$ and the classical BRST complex.
\label{cor:BRSTcoisotropic}
\end{cor}
\begin{proof}
Combining Proposition \ref{prop:hamiltoniancoisotropic} with Theorem \ref{thm:coisotropicintersection}, we see that
\[\C^\bullet(\g, k)\otimes_{\C^\bullet(\g, \Sym \g)}^{\bL} \C^\bullet(\g, B)\]
carries a $\widehat{\bP}_1$-structure.

The two-sided bar complex $k\otimes_{\Sym\g}^\bL B$ is the geometric realization of the simplicial complex $V_\bullet$ where
\[V_n = k\otimes (\Sym \g)^{\otimes n}\otimes B.\]

We also denote by $W^1_\bullet$ the simplicial complex whose geometric realization is \[\Sym(\g^*[-1])\otimes^\bL_{\Sym(\g^*[-1])} \Sym(\g^*[-1])\] and by $W^2_\bullet$ the constant simplicial complex with $W^2_0=\Sym(\g^*[-1])$.

The two-sided bar complex $\C^\bullet(\g, k)\otimes_{\C^\bullet(\g, \Sym \g)}^{\bL} \C^\bullet(\g, B)$ is computed as the geometric realization of the simplicial complex $V_\bullet\otimes W^1_\bullet$ with the Chevalley-Eilenberg differential \eqref{eq:CEdifferential}. The multiplication map gives a weak equivalence of simplicial complexes $W^1_\bullet\rightarrow W^2_\bullet$ which extends to a weak equivalence of simplicial complexes $V_\bullet\otimes W^1_\bullet\rightarrow V_\bullet\otimes W^2_\bullet$ which acts as the identity on $V_\bullet$. This implies that the multiplication map gives a quasi-isomorphism of cdgas
\[\C^\bullet(\g, k)\otimes_{\C^\bullet(\g, \Sym \g)}^{\bL} \C^\bullet(\g, B)\rightarrow \C^\bullet(\g, k\otimes_{\Sym \g}^{\bL} B).\]

We have a quasi-isomorphism of $\g$-representations \[\Sym(\g[1])\otimes B\rightarrow k\otimes^{\bL}_{\Sym\g} B\]
given by the symmetrization
\[x_1\wedge \dots\wedge x_n\otimes b\mapsto \sum_{\sigma\in S_n} (-1)^{\epsilon} [x_{\sigma(1)}|\dots|x_{\sigma(n)}|b].\]
This gives a quasi-isomorphism of cdgas
\[\C^\bullet(\g, \Sym(\g[1])\otimes B)\rightarrow \C^\bullet(\g, k\otimes_{\Sym \g}^{\bL} B).\]

Combining these two quasi-isomorphisms we obtain a quasi-isomorphism
\[B//\Sym\g\rightarrow \C^\bullet(\g, \Sym(\g[1])\otimes B)\]
to the classical BRST complex.
\end{proof}

\begin{remark}
We do not know whether the classical BRST complex is quasi-isomorphic to $B//\Sym\g$ as a $\hat{\bP}_1$-algebra for general $\g$. However, let's restrict to the case $\g$ is an abelian Lie algebra.

We have a splitting of the multiplication map \[\Sym(\g^*[-1])\otimes \Sym(\g^*[-1])^{\otimes n}\otimes \Sym(\g^*[-1])\rightarrow \Sym(\g^*[-1])\]
given by sending $x\mapsto x\otimes 1^{\otimes n}\otimes 1$. This gives a splitting
\[\C^\bullet(\g, k\otimes_{\Sym \g}^{\bL} B)\rightarrow \C^\bullet(\g, k)\otimes_{\C^\bullet(\g, \Sym \g)}^{\bL} \C^\bullet(\g, B)=B//\Sym\g.\]

It is easy to check that the composite map
\[\C^\bullet(\g, \Sym(\g[1])\otimes B)\rightarrow \C^\bullet(\g, k\otimes_{\Sym \g}^{\bL} B)\rightarrow B//\Sym\g\]
is compatible with the Poisson structures.
\end{remark}

\section{Brace algebras}

\label{sect:bracealgebras}

In this section we introduce quantum versions of $\bP_2$-algebras called brace algebras introduced by Gerstenhaber and Voronov, see \cite{GV1} and \cite{GV2}. By a theorem of McClure and Smith \cite{MS} the brace operad controlling brace algebras is a model of the chain operad of little disks $\bE_2$.

\subsection{Algebras}

\begin{defn}
A \textit{brace algebra} $A$ is a dga together with brace operations $A\otimes A^{\otimes n}\rightarrow A[-n]$ for $n>0$ denoted by $x\{y_1,\dots,y_n\}$ satisfying the following equations:
\begin{itemize}
\item (Associativity).
\[x\{y_1,\dots,y_n\}\{z_1,\dots,z_m\} = \sum (-1)^{\epsilon} x \{z_1,\dots,z_{i_1},y_1\{z_{i_1+1},\dots\},\dots,y_n\{z_{i_n+1},\dots\},\dots,z_m\},\]
where the sum goes over the locations of the $y_i$ insertions and the length of each $y_i$ brace. The sign is
\[\epsilon = \sum_{p=1}^n(|y_p|+1)\sum_{q=1}^{i_p} (|z_q|+1).\]

\item (Higher homotopies).
\begin{align*}
\d(x\{y_1,\dots,y_n\}) &= (\d x)\{y_1,\dots,y_n\} \\
&+ \sum_i (-1)^{|x| + \sum_{q=1}^{i-1}|y_q| + i}x\{y_1,\dots, \d y_i,\dots,y_n\} \\
&+ \sum_i (-1)^{|x| + \sum_{q=1}^i|y_q| + i+1}x\{y_1,\dots,y_iy_{i+1},\dots,y_n\} \\
&- (-1)^{(|y_1|+1)|x|}y_1\cdot x\{y_2,\dots,y_n\} \\
&- (-1)^{|x|+\sum_{q=1}^{n-1} |y_q| + n}x\{y_1,\dots,y_{n-1}\}\cdot y_n.
\end{align*}

\item (Distributivity).
\[\sum_{k=0}^n (-1)^{|x_2|(\sum_{q=1}^k |y_q|+k)}x_1\{y_1,\dots,y_k\} x_2\{y_{k+1},\dots,y_n\} = (x_1\cdot x_2)\{y_1,\dots,y_n\}.\]
\end{itemize}
\end{defn}
In the axioms we use a shorthand notation $x\{\}\equiv x$.

\begin{remark}
These axioms coincide with the ones in \cite{GV1} if one flips the sign of the differential.
\end{remark}

For instance, the second axiom for $n=1$ is equivalent to
\[xy - (-1)^{|x||y|} yx = (-1)^{|x|}\d(x\{y\}) - (-1)^{|x|}(\d x)\{y\} + x\{\d y\}.\]
In other words, the multiplication is commutative up to homotopy.

One has the \textit{opposite} brace algebra $A^{\op}$ defined as follows. The product on $A^{\op}$ is opposite to that of $A$:
\[a\cdot^{\op} b := (-1)^{|a||b|} b\cdot a\]
while the braces on $A^{\op}$ are defined by
\[x\{y_1,\dots,y_n\}^{\op}=(-1)^{\sum_{i<j} (|y_i|+1)(|y_j|+1)+n} x\{y_n,\dots,y_1\}.\]

\subsection{Modules}

Let $A$ be a brace algebra. We are now going to define modules over such algebras.

\begin{defn}
A \textit{left brace $A$-module} is a dga $M$ together with a dg homomorphism $A\rightarrow M$ and brace operations $M\otimes A^{\otimes n}\rightarrow M[-n]$ denoted by $m\{x_1,\dots, x_n\}$ satisfying the following equations:
\begin{itemize}
\item (Compatibility). For any $x,y_i\in A$ one has
\[(x\cdot 1)\{y_1,\dots, y_n\} = x\{y_1,\dots, y_n\}\cdot 1.\]

\item (Associativity). For any $m\in M$ and $x_i,y_i\in A$ one has
\begin{align*}
m\{x_1,\dots,x_n\}\{y_1,\dots,y_m\} &= \sum (-1)^{\epsilon}\times \\
&\qquad m\{y_1,\dots,y_{i_1},x_1\{y_{i_1+1},\dots\},\dots,x_n\{y_{i_n+1},\dots\},\dots,y_m\},
\end{align*}
where the sign is
\[\epsilon = \sum_{p=1}^n(|x_p|+1)\sum_{q=1}^{i_p}(|y_q|+1).\]

\item (Higher homotopies). For any $m\in M$ and $x_i\in A$ one has
\begin{align*}
\d(m\{x_1,\dots,x_n\}) &= (\d m)\{x_1,\dots, x_n\} \\
&+\sum (-1)^{|m| + \sum_{q=1}^{i-1} |x_q| + i} m\{x_1, \dots, \d x_i, \dots, x_n\} \\
&+\sum (-1)^{|m| + \sum_{q=1}^i |x_q| + i + 1} m\{x_1, \dots, x_ix_{i+1}, \dots, x_n\} \\
&- (-1)^{|m|(|x_1|+1)}x_1\cdot m\{x_2,\dots, x_n\} \\
&- (-1)^{|m| + \sum_{q=1}^{n-1}|x_q| + n} m\{x_1, \dots, x_{n-1}\}\cdot x_n.
\end{align*}

\item (Distributivity). For any $m,n\in M$ and $x_i\in A$ one has
\[(mn)\{x_1,\dots, x_p\} = \sum_{k=0}^p (-1)^{|n|(\sum_{q=1}^k |x_q| + k)} m\{x_1,\dots, x_k\} n\{x_{k+1},\dots, x_p\}.\]
\end{itemize}
\label{defn:bracemodule}
\end{defn}

\begin{example}
If $A$ is a brace algebra, then it is a left brace $A$-module using the brace operations on $A$ itself.
\end{example}

\begin{remark}
Note that left brace modules are unrelated to the general notion of modules over an algebra over an operad, \cite[Section 12.3.1]{LV}. Our definition is analogous to the notion of a left module over an associative algebra while an operadic module over an associative algebra is a bimodule.
\end{remark}

We define \textit{right brace $A$-modules} to be left brace $A^{\op}$-modules. If $M$ is a left brace $A$-module, then $M^{\op}$ is naturally a right brace $A$-module with the brace operations mirror reversed.

\subsection{Koszul duality}

Let $A$ be a brace algebra. Recall from Section \ref{sect:Koszul} the bar complex $\T_\bullet(A[1])$ which is a dg coalgebra. Since $A$ is not commutative, the shuffle product is not compatible with the differential, so we introduce a slightly different product.

A product
\[\T_\bullet(A[1])\otimes \T_\bullet(A[1])\rightarrow \T_\bullet(A[1])\]
is uniquely specified by the projection to the cogenerators
\[A^{\otimes n}\otimes A^{\otimes m}\rightarrow A[1-n-m].\]

We let the maps with $n=1$ be given by the brace operations and the maps with $n\neq 1$ be zero. Our sign conventions are such that
\[[x]\cdot [y_1|\dots|y_n] = [x\{y_1,\dots,y_n\}]+\dots,\]
i.e. the leading term carries no extra sign.

Extending the product to the whole tensor coalgebra we obtain
\begin{align*}
[x_1|\dots|x_n]\cdot[y_1|\dots|y_m] &= \sum_{\{i_p,l_p\}_{p=1}^n} (-1)^{\epsilon} \times\\
&\qquad [y_1|\dots|y_{i_1}|x_1\{y_{i_1+1},\dots,y_{i_1+l_1}\}|\dots|x_n\{y_{i_n+1},\dots,y_{i_n+l_n}\}|\dots|y_n],
\end{align*}
where the sign is
\[\epsilon = \sum_{p=1}^n(|x_p|+1)\sum_{q=1}^{i_p}(|y_q|+1).\]

\begin{example}
Let $A$ be a commutative algebra considered as a brace algebra with vanishing brace operations. Then the product defined above coincides with the shuffle product.
\end{example}

The following statement is shown in \cite[Lemma 9]{GV2}.

\begin{prop}
Let $A$ be a brace algebra. The multiplication on $\T_\bullet(A[1])$ defined above makes it into a dg bialgebra.
\label{prop:bracekoszul}
\end{prop}
\begin{proof}
By definition the product is compatible with the comultiplication and we only have to check associativity and the Leibniz rule for $\d$.

It is enough to check the components of the identities landing in $A[1]$.

\begin{itemize}
\item (Associativity). The equation
\[([x]\cdot[y_1|\dots|y_n])\cdot[z_1|\dots|z_m] = [x]\cdot ([y_1|\dots|y_n]\cdot[z_1|\dots|z_m])\]
has the following $A$ component:
\begin{align*}
&x\{y_1,\dots,y_n\}\{z_1,\dots,z_m\} \\
&\qquad= \sum_{\{i_p,l_p\}_{p=1}^n} (-1)^{\epsilon}x\{z_1,\dots,z_{i_1},y_1\{z_{i_1+1},\dots,z_{i_1+l_1}\},\dots,y_n\{z_{i_n+1},\dots,z_{i_n+l_n}\},\dots,z_n\}.
\end{align*}
This exactly coincides with the associativity property for brace algebras.

If we replace $[x]$ by $[x_1|\dots|x_m]$ for $m>1$, the associativity equation will have a trivial $A$ component.

\item (Derivation). The equation
\[\d([x]\cdot[y_1|\dots|y_n]) = [\d x] \cdot [y_1|\dots|y_n] + (-1)^{|x|+1} [x]\cdot \d[y_1|\dots|y_n]\]
has the following $A$ component:
\begin{align*}
\d(x\{y_1,\dots, y_n\}) &+ (-1)^{|x|+\sum_{q=1}^{n-1}|y_q| + n} x\{y_1,\dots,y_{n-1}\}\cdot y_n = \\
&- (-1)^{(|y_1|+1)|x|}y_1\cdot x\{y_2,\dots,y_n\} + (\d x)\{y_1,\dots, y_n\} \\
&+ \sum_{i=1}^n (-1)^{\sum_{q=1}^{i-1}|y_q|+|x|+i} x\{y_1,\dots\d y_i,\dots, y_n\} \\
&+ \sum_{i=1}^{n-1}(-1)^{\sum_{q=1}^i|y_q|+|x|+i+1}x\{y_1,\dots,y_iy_{i+1},\dots,y_n\}.
\end{align*}
This follows from the higher homotopy identities for brace algebras.

The equation
\[\d([x_1|x_2]\cdot[y_1|\dots|y_n]) = \d[x_1|x_2] \cdot [y_1|\dots|y_n] + (-1)^{|x_1|+|x_2|} [x_1|x_2]\cdot \d[y_1|\dots|y_n]\]
has the following $A$ component:
\begin{align*}
&\sum_{m=0}^n (-1)^{(|x_2|+1)(\sum_{q=1}^m |y_q| + m)} (-1)^{|x_1|+\sum_{q=1}^m |y_q| + m + 1}x_1\{y_1,\dots,y_m\}x_2\{y_{m+1},\dots,y_n\} \\
&\qquad= (-1)^{|x_1|+1}(x_1 x_2)\{y_1,\dots,y_n\}.
\end{align*}
This follows from the distributivity property for brace algebras.

If we instead have $[x_1|\dots|x_m]$ for $m>2$, this equation will have a trivial $A$ component.
\end{itemize}
\end{proof}

\begin{remark}
It is not difficult to see that $\T_\bullet(A[1])^{\cop}\cong \T_\bullet(A^{\op}[1])$ under the isomorphism \eqref{eq:cofreecop}. Here $(\dots)^{\cop}$ refers to the same dg algebra with the opposite coproduct and $A^{\op}$ is the opposite brace algebra.
\end{remark}

Let us move on to a relative version of this statement. Let $A$ be a brace algebra as before and $M$ a left brace $A$-module. Recall the differential on the bar complex $\T_\bullet(A[1])\otimes M$. We are going to define a dg algebra structure on $\T_\bullet(A[1])\otimes M$ compatibly with the left coaction of $\T_\bullet(A[1])$ such that $M$ and $\T_\bullet(A[1])$ are subalgebras. Thus, we just need to define a braiding morphism
\[M\otimes \T_\bullet(A[1])\rightarrow \T_\bullet(A[1])\otimes M.\]
Compatibility with the $\T_\bullet(A[1])$-comodule structure allows one to uniquely reconstruct this map from the composite
\[M\otimes \T_\bullet(A[1])\rightarrow \T_\bullet(A[1])\otimes M\rightarrow M.\]

We define it using the brace $A$-module structure on $M$. That is, the product is given by
\[[m]\cdot[x_1|\dots|x_n|1] = \sum_{i=0}^n (-1)^{|m|(\sum_{q=1}^i |x_q|+i)}[x_1|\dots|x_i|m\{x_{i+1},\dots, x_n\}].\]

\begin{prop}
Let $M$ be a left brace $A$-module. The previous formula defines a dga structure on $\T_\bullet(A[1])\otimes M$ compatibly with the left $\T_\bullet(A[1])$-comodule structure.
\end{prop}
\begin{proof}
By construction the product on $\T_\bullet(A[1])\otimes M$ is compatible with the $\T_\bullet(A[1])$-coaction, so we just need to check the associativity of the product and the derivation property of $\d$. Due to the compatibility with the $\T_\bullet(A[1])$-coaction, it is enough to check the properties after projection to $M$.

\begin{itemize}
\item (Associativity). The equation
\[[mn]\cdot [x_1|\dots|x_p|1] = [m]\cdot ([n]\cdot [x_1|\dots|x_p|1])\]
has the $M$ component identified with the distributivity property of left brace modules.

Similarly, the equation
\[[m]\cdot ([x_1|\dots|x_n|1]\cdot [y_1|\dots|y_m|1]) = ([m]\cdot [x_1|\dots|x_n|1])\cdot [y_1|\dots|y_m|1]\]
has the $M$-component identified with the associativity property of left brace modules.

\item (Derivation). The equation
\[\d([m]\cdot [x_1|\dots|x_n|1]) = [\d m]\cdot[x_1|\dots|x_n|1] + (-1)^{|m|}[m]\cdot \d[x_1|\dots|x_n|1]\]
has the $M$ component identified with the higher homotopy identities of left brace modules.
\end{itemize}
\end{proof}

We have the same statement for right brace $A$-modules. Indeed, one can replace $A$ by $A^{\op}$ in the previous proposition and observe that the bar complexes $\T_\bullet(A[1])^{\cop}\otimes M$ and $M\otimes \T_\bullet(A[1])$ are isomorphic.

We can combine left and right modules as follows.

\begin{thm}
Let $A$ be a brace algebra, $M$ a left brace $A$-module and $N$ a right brace $A$-module. Then the intersection $N\otimes^{\bL}_A M$ carries a natural dga structure so that the projection $N^{\op}\otimes M\rightarrow N\otimes^{\bL}_A M$ is a morphism of dg algebras.
\label{thm:quantumintersection}
\end{thm}
\begin{proof}
By Proposition \ref{prop:bracekoszul} the bar complex $\T_\bullet(A[1])$ is a dg bialgebra.

Now let $\tilde{M}=\T_\bullet(A[1])\otimes M$ and $\tilde{N}=N\otimes \T_\bullet(A[1])$. By the previous proposition $\tilde{M}$ is a left $\T_\bullet(A[1])$-comodule while $\tilde{N}$ is a right $\T_\bullet(A[1])$-comodule.

The two-sided bar complex $N\otimes^{\bL}_A M$ is isomorphic to the cotensor product $\tilde{N}\otimes^{\T_\bullet(A[1])} \tilde{M}$. As both $\tilde{N}$ and $\tilde{M}$ are dg algebras which are compatible with the coaction of $\T_\bullet(A[1])$, their cotensor product is also a dga.
\end{proof}

\begin{remark}
Given a model for the derived tensor product of the right $A$-module $N$ and a left $A$-module $M$, it is quasi-isomorphic to the two-sided bar complex $N\otimes^\bL_A M$, so by homotopy transfer one can induce a homotopy associative structure on the given model.
\end{remark}

\section{Quantum Hamiltonian reduction}

\subsection{Hochschild cohomology}

Let $A$ be a dga and $B$ an $A$-bimodule. We define the Hochschild cochain complex $\CC^\bullet(A, B)$ to be the graded vector space
\[\CC^\bullet(A, B) = \bigoplus_{n=0}^\infty \Hom(A^{\otimes n}, B)[-n]\]
with the differential
\begin{align*}
(\d f)(x_1,\dots, x_n) &= \d f(x_1, \dots, x_n) \\
&+\sum_{i=1}^n (-1)^{|f| + \sum_{q=1}^{i-1} |x_q|+i+1} f(x_1, \dots, \d x_i, \dots, x_n) \\
&+\sum_{i=1}^{n-1} (-1)^{|f| + \sum_{q=1}^i |x_q|+i} f(x_1, \dots, x_ix_{i+1}, \dots, x_n) \\
&+ (-1)^{|f|(|x_1|+1)} x_1 f(x_2, \dots, x_n) + (-1)^{\sum_{q=1}^{n-1} |x_q| + |f| + n} f(x_1, \dots, x_{n-1}) x_n.
\end{align*}

Given two $A$-bimodules $B_1$ and $B_2$ we have a cup product map
\[\CC^\bullet(A, B_1)\otimes \CC^\bullet(A, B_2)\rightarrow \CC^\bullet(A, B_1\otimes B_2)\]
given by
\[(f_1\smile f_2)(x_1,\dots, x_n) = \sum_{i=0}^n (-1)^{|f_2|(\sum_{q=1}^i |x_q| + i)}f_1(x_1,\dots, x_i) \otimes f_2(x_{i+1}, \dots, x_n).\]

A relation between Hochschild and Chevalley--Eilenberg cohomology is given by the following construction. Let $V$ be a $\U\g$-bimodule. Then $V^{ad}$ is a $\g$-representation with the action given by
\[x.v:= xv - (-1)^{|x||v|}vx,\quad x\in\g,\quad v\in V.\]

Consider $f\in\Hom((\U\g)^{\otimes n}, V)[-n]\subset \CC^\bullet(\U\g, V)$. We get an element $\tilde{f}\in \C^\bullet(\g, V)$ by the following formula:
\[\tilde{f}(x_1,\dots, x_n) = (-1)^{\sum_{q=1}^{n-1}(n-q)|x_q|} \sum_{\sigma\in S_n} (-1)^{\epsilon} f(x_{\sigma(1)}, \dots, x_{\sigma(n)}),\]
where $\epsilon$ is given by the Koszul sign rule with $x_i$ in degree $|x_i|+1$.

The following theorem can be found in \cite[Theorem 2.5]{CR}.
\begin{prop}
Let $A=\U\g$ and $V$ be a $\U\g$-bimodule. Then the morphism
\[\CC^\bullet(\U\g, V)\rightarrow \C^\bullet(\g, V)\]
we have defined is a quasi-isomorphism. Moreover, it is compatible with cup products.
\label{prop:CEvsHH}
\end{prop}

\subsection{Hochschild cohomology and braces}
Gerstenhaber and Voronov \cite{GV1} observed that the Hochschild cochain complex $\CC^\bullet(A, A)$ is a brace algebra which was the motivating example. We define the brace operations as follows:
\begin{align}
&x\{x_1,\dots,x_n\}(a_1,\dots,a_m) \nonumber \\
&\qquad= \sum (-1)^{\epsilon} x(a_1,\dots,a_{i_1},x_1(a_{i_1+1},\dots,a_{i_1+l_1}),\dots,x_n(a_{i_n+1},\dots,a_{i_n+l_n}),\dots,a_m),
\label{eq:HHbrace}
\end{align}
where the sign is determined by the following rule: $x_i$ moving past $a_j$ produces the sign $(|x_i|+1)(|a_j|+1)$.

A multiplication on $A$ determines a degree 2 element $m$ of $\CC^\bullet(A, A)$ via \[m(x, y) = (-1)^{|x|+1}xy.\] The differential on $\CC^\bullet(A, A)$ is the sum of the natural differential on $\oplus_n \Hom((A[1])^{\otimes n}, A)$ and the differential $m\{f\} + (-1)^{|f|} f\{m\}$. The cup product on $\CC^\bullet(A, A)$ is given by the formula $f_1\smile f_2 = (-1)^{|f_1|+1} m\{f_1, f_2\}$.

We will also need a variation of this example. Let $B$ be a dga and $\mu\colon A\rightarrow B$ a morphism. Using the brace operations as above, one can turn $\CC^\bullet(A, B)$ into a left brace $\CC^\bullet(A, A)$-module.

We can also use the Hochschild cochain complex to give an interpretation of brace modules similar to Definition \ref{defn:coisotropic}.

\begin{prop}
Let $A$ be a brace algebra and $M$ a left brace $A$-module with the module structure given by a morphism of algebras $f_0\colon A\rightarrow M$. Then we have a lift
\[
\xymatrix{
\T_\bullet(A[1]) \ar^{f_0}[r] \ar@{-->}_{f}[dr] & \T_\bullet(M[1]) \\
& \T_\bullet(\CC^\bullet(M, M)[1]), \ar[u]
}
\]
where $f$ is a morphism of dg bialgebras.
\end{prop}
\begin{proof}
A morphism of coalgebras $f\colon \T_\bullet(A[1])\rightarrow \T_\bullet(\CC^\bullet(M, M)[1])$ is uniquely specified by the composite
$\T_\bullet(A[1])\rightarrow \T_\bullet(\CC^\bullet(M, M)[1])\rightarrow \CC^\bullet(M, M)[1]$
which consists of morphisms
\[f_{m,n}\colon M^{\otimes m}\otimes A^{\otimes n}\rightarrow M[1-n-m].\]

We define $f_{m,n} = 0$ for $m>1$. The operations $f_{1,n}$ are given by
\[f_{1,n}(m,x_1,\dots,x_n) = m\{x_1,\dots, x_n\}.\]

A straightforward computation shows that the first two axioms in Definition \ref{defn:bracemodule} are equivalent to the compatibility of $f$ with the multiplications and the last two axioms are equivalent to the compatibility of $f$ with the differentials.
\end{proof}

\begin{remark}
A triple $(A, M, f)$ of a brace algebra $A$, a dga $M$ and a morphism of brace algebras $f\colon A\rightarrow \CC^\bullet(M, M)$ is expected (see \cite[Section 2.5]{Kon}) to be the same as an algebra over chains on the two-dimensional Swiss-cheese operad. The corresponding statement in the topological setting has been proved in \cite{Th}. Partial progress has been made in \cite{DTT} where the authors show that the pair $(\CC^\bullet(M, M), M)$ is indeed an algebra over the Swiss-cheese operad.
\end{remark}

\subsection{Hamiltonian reduction}

Let $B$ be a dg algebra with a $\g$-action. We denote by
\[a\colon\g\rightarrow\Der(B)\]
the action morphism. Under deformation quantization the notion of a moment map for Poisson algebras (Definition \ref{def:momentmap}) is deformed as follows.

\begin{defn}
A $\g$-equivariant morphism $\mu\colon \g\rightarrow B$ is a \textit{quantum moment map} if the equation
\[[\mu(x), b] = a(x).b\]
is satisfied for all $x\in\g$ and $b\in B$.
\end{defn}

We refer to \cite{Et} for details on quantum moment maps.

\begin{remark}
As in the case of classical moment maps, one can replace $\g$-equivariance by the condition that $\mu$ extends to a morphism of dg algebras $\U\g\rightarrow B$.
\end{remark}

\begin{defn}
Suppose $B$ is a dga equipped with a $\g$-action and a quantum moment map $\mu\colon \U\g\rightarrow B$. Its 
\textit{quantum Hamiltonian reduction} is
\[B//\U\g= \CC^\bullet(\U\g, k)\otimes^{\bL}_{\CC^\bullet(\U\g, \U\g)} \CC^\bullet(\U\g, B).\]
\label{defn:qhamreduction}
\end{defn}

In this bar complex we use the left $\CC^\bullet(\U\g, \U\g)$-module structure on $\CC^\bullet(\U\g, B)$ coming from the moment map $\U\g\rightarrow B$ and the right $\CC^\bullet(\U\g, \U\g)$-module structure on $\CC^\bullet(\U\g, k)$ coming from the counit. We put a dga structure on $B//\U\g$ in Corollary \ref{cor:quantumBRSTtensor}.

There is a quantum version of the BRST complex introduced in \cite{KS}. As a complex, it has the following description. We will assume that the Lie algebra $\g$ is unimodular, i.e. the representation $\det(\g)$ is trivial.

Recall the Koszul complex $\Sym(\g[1])\otimes B$ that we have defined in Section \ref{sect:classicalhamreduction}. We are going to deform it to the Chevalley--Eilenberg differential as follows. Given
\[x_1\wedge \dots\wedge x_n\otimes b\in\Sym(\g[1])\otimes B\] we let
\begin{align}
\d(x_1\wedge \dots\wedge x_n\otimes b) &= \sum_{i=1}^n (-1)^{(|x_i|+1)(\sum_{q=1}^{i-1} |x_q| + i - 1)} \d x_i\wedge x_1\wedge \dots\wedge \hat{x}_i\wedge \dots\wedge x_n\otimes b \nonumber \\
&- \sum_{i=1}^n (-1)^{|x_i|\sum_{q=i+1}^n(|x_q|+1) + \sum_{q=1}^{i-1}(|x_q|+1)+|x_i|} x_1\wedge \dots\wedge \hat{x}_i\wedge \dots\wedge x_n\otimes \mu(x_i)b \nonumber \\
&+ (-1)^{\sum_{q=1}^n |x_q| + n} x_1\wedge \dots\wedge x_n \otimes \d b \nonumber \\
&+ \sum_{i<j} (-1)^{(|x_i|+1)(\sum_{q=1}^{i-1}|x_q|+i) + (|x_j|+1)(\sum_{q=1,q\neq i}^{j-1}|x_q|+j-1)} \times \nonumber \\
&\qquad\qquad \times [x_i, x_j]\wedge x_1\wedge \dots, \widehat{x}_i,\dots,\widehat{x}_j,\dots,x_n\otimes b.
\label{eq:quantumkoszul}
\end{align}

The quantum BRST complex is then
\[\C^\bullet(\g, \Sym(\g[1])\otimes B).\]

We refer the reader to \cite[Section 6]{KS} for a detailed description of the quantum BRST complex together with a dga structure.

\subsection{Hamiltonian reduction as an intersection}

Let $B$ be a dga with a Hamiltonian action of $\g$. Recall that $\CC^\bullet(\U\g, B)$ is then a left brace module over $\CC^\bullet(\U\g, \U\g)$. Similarly, $\CC^\bullet(\U\g, k)$ is a left brace module using the counit map $\U\g\rightarrow k$ and hence $\CC^\bullet(\U\g, k)^{\op}$ is a right brace module. Using Theorem \ref{thm:quantumintersection} we therefore have a natural multiplication on the tensor product of $\CC^\bullet(\U\g, k)$ and $\CC^\bullet(\U\g, B)$.

\begin{cor}
The quantum Hamiltonian reduction
\[B//\U\g= \CC^\bullet(\U\g, k)\otimes^{\bL}_{\CC^\bullet(\U\g, \U\g)} \CC^\bullet(\U\g, B)\]
carries a natural dga structure. Moreover, it is quasi-isomorphic to the quantum BRST complex.
\label{cor:quantumBRSTtensor}
\end{cor}
\begin{proof}
The zig-zag of quasi-isomorphisms mentioned in the statement of the theorem is as follows:
\[
\xymatrix{
\CC^\bullet(\U\g, k)\otimes^{\bL}_{\CC^\bullet(\U\g, \U\g)} \CC^\bullet(\U\g, B) \ar[d] \\
\CC^\bullet(\U\g, k\otimes^{\bL}_{\U\g}B) \\
\CC^\bullet(\U\g, \Sym(\g[1])\otimes B) \ar[u] \ar[d] \\
\C^\bullet(\g, \Sym(\g[1])\otimes B).
}
\]

\begin{itemize}
\item The morphism
\[\CC^\bullet(\U\g, k)\otimes^{\bL}_{\CC^\bullet(\U\g, \U\g)} \CC^\bullet(\U\g, B)\rightarrow \CC^\bullet(\U\g, k\otimes^{\bL}_{\U\g}B)\]
is given by the cup product. The fact that it is a quasi-isomorphism is proved as in Corollary \ref{cor:BRSTcoisotropic}.

\item The morphism
\[\CC^\bullet(\U\g, \Sym(\g[1])\otimes B)\hookrightarrow \CC^\bullet(\U\g, k\otimes^{\bL}_{\U\g}B)\]
is given by including the Chevalley--Eilenberg chain complex into the bar complex.

\item The morphism
\[\CC^\bullet(\U\g, \Sym(\g[1])\otimes B)\rightarrow \C^\bullet(\g, \Sym(\g[1])\otimes B)\]
is the restriction morphism which is a quasi-isomorphism by Proposition \ref{prop:CEvsHH}.
\end{itemize}
\end{proof}

\begin{remark}
As for the classical BRST complex, we do not know if the quasi-isomorphism above is compatible with the multiplication.
\end{remark}

\subsection{\texorpdfstring{$\bE_n$}{En} Hamiltonian reduction}

\label{sect:Enreduction}

The interpretation of quantum Hamiltonian reduction as a tensor product of brace modules allows one to formulate an $\bE_n$ version of quantum Hamiltonian reduction. In this section we sketch what such a notion looks like in the $\infty$-categorical setting. We refer to \cite{Gin} for some basics of $\bE_n$-algebras that we will use.

Let $\bE_n$ be the chain operad of little $n$-cubes. For instance, the operad $\bE_1$ is quasi-isomorphic to the associative operad and $\bE_2$ is quasi-isomorphic to the brace operad. Given a morphism of $\bE_n$-algebras $f\colon A\rightarrow B$ one has the \textit{$\bE_n$-centralizer} $\Z(f)$ which is an $\bE_n$-algebra satisfying a certain universal property \cite[Definition 24]{Gin}. For $f=\id\colon A\rightarrow A$ we denote $\Z(\id)=\Z(A)$, the center of $A$, which is an associative algebra object in $\bE_n$-algebras, i.e. an $\bE_{n+1}$-algebra by Dunn--Lurie additivity \cite[Theorem 5.1.2.2]{HA}. Note that in the case of associative algebras (i.e. $n=1$), $\Z(A)$ coincides with the Hochschild complex and its zeroth cohomology is the center of $A$ in the usual sense.

One has a forgetful functor from $\bE_n$-algebras to Lie algebras which on the level of underlying complexes is $A\mapsto A[n-1]$. The left adjoint to this forgetful functor is called the universal enveloping $\bE_n$-algebra functor and is denoted by $\U_{\bE_n}$, see \cite[Section 7.5]{Gin}.

Let $B$ be an $\bE_n$-algebra with an action of the Lie algebra $\g$, i.e. we have a morphism of Lie algebras $a\colon\g\rightarrow \bT_B$ to the tangent complex of $B$.

\begin{defn}
A \textit{quantum moment map} for the $\g$-action on $B$ is a morphism of Lie algebras $\g\rightarrow B[n-1]$ fitting into the diagram
\[
\xymatrix{
B[n-1] \ar[r] & \bT_B \\
& \g \ar^{a}[u] \ar@{-->}[ul]
}
\]
of Lie algebras.
\end{defn}

By adjunction the morphism of Lie algebras $\g\rightarrow B[n-1]$ gives rise to a morphism of $\bE_n$-algebras $\mu\colon\U_{\bE_n}(\g)\rightarrow B$. By the defining property of centralizers we see that $\Z(\mu)$ is a left module over the $\bE_{n+1}$-algebra $\Z(\U_{\bE_n}(\g))$ in the $\infty$-category of $\bE_n$-algebras. Alternatively, we can view the pair $(\Z(\U_{\bE_n}(\g)), \Z(\mu))$ as an $\bE_n$-algebra in the $\infty$-category $\mathrm{LMod}$ of pairs of an associative algebra and a module.

\begin{remark}
Note that using \cite[Theorem 14]{Gin} one can identify $\Z(\mu)\cong \C^\bullet(\g, B)$.
\end{remark}

Let $\epsilon\colon \U_{\bE_n}\g\rightarrow k$ be the counit map. If $n>0$ one can choose an isomorphism \[\Z(\U_{\bE_n}\g)\cong\Z(\U_{\bE_n}\g)^{\op}\] making $\Z(\epsilon)$ into a right module over $\Z(\U_{\bE_n}\g)$. Thus, if we denote by $\mathrm{BiMod}$ the $\infty$-category of triples $(A, M, N)$ of an associative algebra $A$, a left $A$-module $M$ and a right $A$-module $N$, then we see that the triple $(\Z(\U_{\bE_n}\g), \Z(\mu), \Z(\epsilon))$ becomes an $\bE_n$-algebra in $\mathrm{BiMod}$. In particular, for any $\bE_n$-algebra $(A, M, N)$ in $\mathrm{BiMod}$, the bar construction $N\otimes_A M$ is still an $\bE_n$-algebra.

\begin{defn}
Let $B$ be an $\bE_n$-algebra with a $\g$-action and a moment map $\mu\colon \U_{\bE_n}(\g)\rightarrow B$. Its \textit{$\bE_n$ Hamiltonian reduction} is the $\bE_n$-algebra
\[B//\U_{\bE_n}(\g)=\Z(\epsilon)\otimes_{\Z(\U_{\bE_n}\g)} \Z(\mu).\]
\end{defn}

For instance, consider the case $n=1$. Then $\U_{\bE_1}(\g)$ coincides with the usual enveloping algebra. Indetifying centralizers with the Hochschild complex, we get the formula
\[B//\U(\g) = \CC^\bullet(\U\g, k)\otimes_{\CC^\bullet(\U\g, \U\g)} \CC^\bullet(\U\g, B)\]
recovering quantum Hamiltonian reduction given in Definition \ref{defn:qhamreduction}.

\section{Classical limits}

\label{sect:classicallimits}

In this section we relate some constructions in Section \ref{sect:Pnalgebras} to those in Section \ref{sect:bracealgebras}. Namely, we formulate precisely in which sense constructions in Section \ref{sect:bracealgebras} are quantizations. Along the way we also relate Baranovsky and Ginzburg's construction \cite{BG} of the Poisson structure on a coisotropic intersection to our formulas.

\subsection{Beilinson--Drinfeld algebras}

A precise sense in which associative algebras are quantizations of Poisson algebras is given by Beilinson--Drinfeld ($\BD$) algebras \cite[Section 2.4]{CG}. Let us recall the definition.

\begin{defn}
A \textit{$\BD_1$-algebra} is a dgla $A$ over $k\llbracket\hbar\rrbracket$ together with an associative $k\llbracket\hbar\rrbracket$-linear multiplication satisfying the relations
\begin{itemize}
\item $\hbar\{x, y\} = xy - (-1)^{|x||y|}yx$,

\item $\{x, yz\} = \{x, y\}z + (-1)^{|x||y|} y\{x, z\}$.
\end{itemize}
\end{defn}

To understand this definition, recall that dg algebras are naturally Lie algebras with the bracket given by the commutator. The notion of a $\BD_1$-algebra then captures the fact that the Lie bracket vanishes to the first order at $\hbar=0$. In the classical limit we have an isomorphism of operads
\[\BD_1/\hbar\cong \bP_1\]
while in the quantum case $\hbar\neq 0$ we have
\[\BD_1[\hbar^{-1}]\cong \Ass\otimes_k k\llpar\hbar\rrpar\]
since the bracket is then uniquely determined from the multiplication. In other words, the operad $\BD_1$ interpolates between the Poisson operad $\bP_1$ and the associative operad $\Ass$.

\begin{remark}
One can show that $\BD_1(n)$ is free as a $k\llbracket\hbar\rrbracket$-module.
\end{remark}

Let us also mention that there is a canonical isomorphism of operads
\[\bP_1\otimes k[\hbar]/\hbar^2\stackrel{\sim}\rightarrow \BD_1/\hbar^2\]
given by sending the multiplication to $\frac{ab+(-1)^{|a||b|}ba}{2}$.

Given a $\BD_1$-algebra $A$, we let $A^{\op}$ be the opposite algebra with the operations
\begin{align*}
a\cdot^{\op} b &= (-1)^{|a||b|} b\cdot a \\
\{a, b\}^{\op} &= -\{a, b\}.
\end{align*}

There are also lower-dimensional and higher-dimensional versions of the $\BD_n$ operad.

\begin{defn}
A \textit{$\BD_0$-algebra} is a complex $A$ over $k\llbracket\hbar\rrbracket$ together with a degree $1$ Lie bracket and a unital commutative multiplication satisfying the relations
\begin{itemize}
\item $\d(ab) = \d(a)b + (-1)^{|a|} a \d(b) + \hbar\{a, b\}$,

\item $\{x, yz\} = \{x, y\}z + (-1)^{|y||z|}\{x, z\}y$.
\end{itemize}
\end{defn}

In the classical limit we have an isomorphism
\[\BD_0/\hbar\cong \bP_0\]
since then the multiplication is compatible with the differential. In the quantum case $\hbar\neq 0$ we have
\[\BD_0[\hbar^{-1}]\cong \widehat{\bE}_0\otimes k\llpar\hbar\rrpar,\]
where the operad $\widehat{\bE}_0$ is contractible, i.e. quasi-isomorphic to the operad $\bE_0$ controlling complexes with a distinguished vector.

Let $\coBD_1$ be the cooperad obtained as the $k\llbracket\hbar\rrbracket$-linear dual to the operad $\BD_1$. It has a natural Hopf structure, so following Calaque and Willwacher \cite{CW} we can consider its brace construction
\[\BD_2 = \Br_{\coBD_1}.\]

By construction we have that
\[\BD_2[\hbar^{-1}] = \Br_{\coBD_1[\hbar^{-1}]}\cong \Br_{\coAss}\otimes k\llpar\hbar\rrpar,\]
where $\Br_{\coAss}$ is a dg operad quasi-isomorphic to the operad $\Br$ controlling brace algebras, but where the product is merely $A_\infty$.

Moreover,
\[\BD_2/\hbar \cong \Br_{\coP_1}\cong \bP_2,\]
where the last quasi-isomorphism is given by \cite[Theorem 4]{CW} (we remind the reader that $e_n\cong \bP_n$ for $n\geq 2$).

\subsection{Modules}

Let us now describe modules over $\BD_1$-algebras.

Recall that a coisotropic morphism $A\rightarrow B$ for $A$ a $\bP_1$-algebra is the data of a $\bP_0$-algebra on $B$ and a morphism of $\bP_1$-algebras
\[A\rightarrow \Z(B)\cong \widehat{\Sym}(\T_B),\]
where for simplicity we have assumed that $\Omega^1_B$ is dualizable as a dg module over $B$.

For $B$ a commutative graded algebra we denote by $\widehat{\D}_{\hbar}(B)$ the completed algebra of $\hbar$-differential operators. That is, it is an algebra over $k\llbracket\hbar\rrbracket$ generated by elements of $B$ and $\T_B$ with the relations
\begin{align*}
vw - (-1)^{|v||w|}wv &= \hbar[v, w],\quad v,w\in \T_B \\
vb - (-1)^{|v||b|}bv &= \hbar v.b,\quad v\in \T_B,\ w\in B
\end{align*}
completed with respect to the increasing filtration given by the order of differential operators.

If $B$ is a $\BD_0$-algebra, the data of the differential on $B$ determines a Maurer--Cartan element in $\widehat{\D}_{\hbar}(B)$ and we denote by $\Z(B)$, the \textit{$\BD_0$-center} of $B$, the algebra $\widehat{\D}_{\hbar}(B)$ with the differential twisted by that Maurer--Cartan element. It is clear that $\Z(B)$ is a $\BD_1$-algebra.

More generally, if $B$ is a commutative graded algebra, the data of a Maurer--Cartan element in $\widehat{\D}_{\hbar}(B)$ will be called a $\widehat{\BD}_0$-algebra structure on $B$. Note that $\BD_0$-structures correspond to those Maurer--Cartan elements which have order at most 2.

\begin{remark}
Suppose $B_0$ is a cofibrant commutative dg algebra over $k$. We can trivially extend it to a $\BD_0$-algebra $B=B_0\otimes k\llbracket\hbar\rrbracket$ with the bracket defined to be zero. Then we expect that $\widehat{\D}_{\hbar}(B)$ coincides with the center of $B\in\Alg_{\BD_0}$ in the sense of \cite[Definition 5.3.1.6]{HA}. Note that given any $\BD_0$-algebra $B$, its $\BD_0$-center at $\hbar=0$ becomes $(\widehat{\Sym}(\T_{B_0}), [\pi_{B_0}, -])$, the $\bP_0$-center of $B_0=B/\hbar$.
\end{remark}

Let $A$ be another $\BD_1$-algebra.

\begin{defn}
A \emph{left $\BD_1$-module} over $A$ is a $\BD_0$-algebra $B$ together with a morphism $A\rightarrow \Z(B)$ of $\BD_1$-algebras.
\end{defn}

A \textit{right $\BD_1$-module} over $A$ is the same as a left $\BD_1$-module over $A^{\op}$.

It is clear that the definition at $\hbar=0$ reduces to the definition of a coisotropic morphism. Thus, one can talk about \textit{quantizations} of a given coisotropic morphism $A_0\rightarrow B_0$: these are $\BD_1$-algebras $A$ and $\BD_0$-algebras $B$ reducing to the given algebras $A_0,B_0$ at $\hbar=0$ together with a left $\BD_1$-module structure on $B$.

One can similarly define $\BD_2$-modules as follows. Given a complex $B$, we denote by $\coBD_1(B)$ the cofree conilpotent $\BD_1$-coalgebra on $B$. Given a $\BD_1$-algebra $B$, we define its \emph{center} to be the complex
\[\Z(B) = \Hom(\coBD_1(B), B)\]
twisted by the differential given by the $\BD_1$-structure on $B$. By the results of \cite{CW}, $\Z(B)$ is a $\BD_2=\Br_{\coBD_1}$-algebra, so we can give the following definition.

Let $A$ be a $\BD_2$-algebra.

\begin{defn}
A \emph{left $\BD_2$-module} over $A$ is a $\BD_1$-algebra $B$ together with a morphism $A\rightarrow \Z(B)$ of $\BD_2$-algebras.
\end{defn}

\begin{remark}
Just like for Poisson algebras, we expect that the $\BD_1$-center of a $\BD_1$-algebra $B\in\Alg_{\BD_1}$ satisfies the universal property of \cite[Definition 5.3.1.6]{HA}. See also Remark \ref{remark:centralizer}.
\end{remark}

\subsection{From \texorpdfstring{$\BD_1$}{BD1} to \texorpdfstring{$\BD_0$}{BD0}}

We are now going to prove a $\BD_1$-version of Theorem \ref{thm:coisotropicintersection}.

Let $A$ be a $\BD_1$-algebra. In particular, it is a dga and so we have a dg coalgebra $\T_\bullet(A[1])$. Introduce a commutative multiplication on $\T_\bullet(A[1])$ given by the shuffle product and the Lie bracket given by \eqref{eq:PnKoszulbracket}.

Since $A$ is not necessarily commutative, the differential is not compatible with the shuffle product. But its failure is exactly captured by the bracket.

\begin{prop}
Let $A$ be a $\BD_1$-algebra. The differential, multiplication and the bracket make $\T_\bullet(A[1])$ into a $\BD_0$-algebra compatibly with the coalgebra structure.
\end{prop}
\begin{proof}
To prove the claim we just need to show that the relation between the differential on $\T_\bullet(A[1])$ and the product is exactly the one that appears in the definition of $\BD_0$-algebras.

Due to the compatibility of the operations with the coproduct on $\T_\bullet(A[1])$, we just need to check the corresponding relation after projection to $A[1]$.

For $a,b\in A$ we have
\begin{align*}
\d([a]\cdot [b]) &= \d([a|b] + (-1)^{(|a|+1)(|b|+1)}[b|a]) \\
&= [\d a|b] + (-1)^{|a|+1} [a|\d b] + (-1)^{(|a|+1)(|b|+1)}[\d b|a] + (-1)^{|a|(|b|+1)}[b|\d a] \\
&+(-1)^{|a|+1}[ab] + (-1)^{|a|(|b|+1)}[ba].
\end{align*}

Similarly, we have
\[
[\d a]\cdot [b] + (-1)^{|a|+1}[a]\cdot [\d b] = [\d a|b] + (-1)^{|a|(|b|+1)}[b|\d a] + (-1)^{|a|+1}[a|\d b] + (-1)^{(|a|+1)(|b|+1)}[\d b|a].
\]

Their difference is given by
\begin{align*}
(-1)^{|a|+1}[ab] + (-1)^{|a|(|b|+1)}[ba] &= \hbar (-1)^{|a|+1}[\{a, b\}] \\
&= \hbar \{[a], [b]\}.
\end{align*}
\end{proof}

We can also add modules in the picture. Let $M$ be a left $\BD_1$-module. Then $\T_\bullet(A[1])\otimes M$ carries a differential and $L_\infty$ brackets given by equations \eqref{eq:barcomplexlk} and \eqref{eq:barcomplexl2}. Moreover, $\T_\bullet(A[1])\otimes M$ carries a natural multiplication.

\begin{prop}
Let $A$ be a $\BD_1$-algebra and $M$ a left $\BD_1$-module. Then $\T_\bullet(A[1])\otimes M$ carries a natural structure of a left $\widehat{\BD}_0$-comodule over $\T_\bullet(A[1])$.
\end{prop}
\begin{proof}
By construction the differential and the brackets on $\T_\bullet(A[1])\otimes M$ are compatible with the $\T_\bullet(A[1])$-comodule structure, so we just have to check that the projection of the differential on $M$ has symbol given by the brackets.

The differential lands in $M$ in the following two cases:
\begin{enumerate}
\item The symbol of $\d\colon M\rightarrow M$ is given by the Poisson bracket on $M$ since it is a $\BD_0$-algebra.

\item The symbol of the action map $A\otimes M\rightarrow M$ is given by the $\hbar=0$ limit of the action map $A\rightarrow \D_\hbar(M)$ which coincides with the Poisson brackets on $\T_\bullet(A[1])\otimes M$ given by formula \eqref{eq:barcomplexlk}.
\end{enumerate}
\end{proof}

Finally, suppose $M$ and $N$ are a left and right $\BD_1$-modules over $A$ respectively. Then on the two-sided bar complex
\[N\otimes \T_\bullet(A[1])\otimes M\]
we can introduce the usual bar differential and the shuffle product.

\begin{thm}
Let $A$ be a $\BD_1$-algebra, $M$ a left $\BD_1$-module and $N$ a right $\BD_1$-module over $A$. Then the two-sided bar complex $N\otimes \T_\bullet(A[1])\otimes M$ has a $\widehat{\BD}_0$-structure.
\end{thm}

At $\hbar=0$ this construction recovers the $\widehat{\bP}_0$-structure of Theorem \ref{thm:coisotropicintersection}.

\end{document}